\documentclass[11pt]{article}
\usepackage{latexsym,amsmath,color,amsthm,amssymb,epsfig,graphicx,mathrsfs}
\usepackage{graphicx}
\usepackage{amssymb}
\usepackage[left=1in,top=1in,right=1in,bottom=1in]{geometry}
\usepackage[linktocpage=true]{hyperref}
\usepackage{setspace}
\usepackage{amssymb, amsmath, amsthm, graphicx,mathrsfs}

\def\qed{\hfill\ifhmode\unskip\nobreak\fi\quad\ifmmode\Box\else\hfill$\Box$\fi}
\def\ite#1{\hfill\break${}$\hbox to 50pt {\quad(#1)\hfill}}
\newtheorem{thm}{Theorem}[section]
\newtheorem{cor}[thm]{Corollary}

\newtheorem{lem}[thm]{Lemma}

\newtheorem{prop}[thm]{Proposition}
\newtheorem{claim}[thm]{Claim}

\def\ex{{\rm{ex}}}

\begin{document}

\title{\vspace{-0.5in}Stability in the Erd\H{o}s--Gallai Theorem on  cycles and paths}

\author{
{{Zolt\'an F\" uredi}}\thanks{
\footnotesize {Alfr\' ed R\' enyi Institute of Mathematics, Hungary
E-mail:  \texttt{zfuredi@gmail.com.}
Research supported in part by the Hungarian National Science Foundation OTKA 104343,
 by the Simons Foundation Collaboration Grant 317487,
and by the European Research Council Advanced Investigators Grant 267195.
}}
\and
{{Alexandr Kostochka}}\thanks{
\footnotesize {University of Illinois at Urbana--Champaign, Urbana, IL 61801
 and Sobolev Institute of Mathematics, Novosibirsk 630090, Russia. E-mail: \texttt {kostochk@math.uiuc.edu}.
 Research of this author
is supported in part by NSF grants  DMS-1266016 and  DMS-1600592
and  by Grant NSh 1939.2014.1 of the President of
Russia for Leading Scientific Schools.
}}
\and{{Jacques Verstra\"ete}}\thanks{Department of Mathematics, University of California at San Diego, 9500
Gilman Drive, La Jolla, California 92093-0112, USA. E-mail: {\tt jverstra@math.ucsd.edu.} Research supported by NSF Grant DMS-1101489. }}

\date{\small July 19, 2015. Revised on May 9, 2016}

\maketitle

\vspace{-0.3in}

\begin{center}
 Dedicated to the memory of G. N. Kopylov
\end{center}

\begin{abstract}
The Erd\H{o}s-Gallai Theorem states that for $k \geq 2$, every graph of average degree more than $k - 2$ contains a $k$-vertex path.
This result is a consequence of a stronger result of Kopylov: if $k$ is odd, $k=2t+1\geq 5$, 
  $n \geq (5t-3)/2$,
and $G$ is an $n$-vertex $2$-connected graph
with at least $h(n,k,t) := {k-t \choose 2} + t(n -k+ t)$ edges, then $G$ contains a cycle of length at least $k$ unless $G = H_{n,k,t} := K_n - E(K_{n - t})$.

In this paper we prove a stability version of the Erd\H{o}s-Gallai Theorem: we show that
for all $n \geq 3t > 3$, and $k \in \{2t+1,2t + 2\}$, every $n$-vertex 2-connected graph $G$ with $e(G) > h(n,k,t-1)$ either contains a cycle of length at least $k$ or contains a set of $t$ vertices whose removal gives a star forest.
In particular, if $k = 2t + 1 \neq 7$, we show $G \subseteq  H_{n,k,t}$. The lower bound $e(G) > h(n,k,t-1)$ in these results is tight
   and  is smaller than Kopylov's bound $h(n,k,t)$ by a term of $n-t-O(1)$.

\end{abstract}

\section{Introduction}

A cornerstone of extremal combinatorics is the study of Tur\'{a}n-type problems for graphs. One of the fundamental questions in extremal graph theory
is to determine the maximum number of edges in an $n$-vertex graph with no $k$-vertex path. According to~\cite{FS224}, this problem was posed by Tur\'{a}n. A solution
to the problem was obtained by Erd\H{o}s and Gallai~\cite{ErdGal59}:

\begin{thm}[Erd\H{o}s and Gallai~\cite{ErdGal59}]\label{ErdGallaiPath}
Let $G$ be an $n$-vertex graph with more than $\frac{1}{2}(k-2)n$ edges, $k\ge 2$.
Then $G$ contains a $k$-vertex path $P_k$.
\end{thm}

This result is best possible for $n$ divisible by $k - 1$, due to the $n$-vertex graph whose components are cliques of order $k - 1$.
To obtain Theorem \ref{ErdGallaiPath}, Erd\H{o}s and Gallai observed that if $H$ is an $n$-vertex graph without a $k$-vertex path $P_k$, then adding a new vertex and joining it to all other vertices we have a graph $H'$
on $n+1$ vertices $e(H)+n$ edges and containing no cycle $C_{k+1}$ or longer. Then Theorem~\ref{ErdGallaiPath} is a consequence of the following:

\begin{thm}[Erd\H{o}s and Gallai~\cite{ErdGal59}]\label{ErdGallaiCyc}
Let $G$ be an $n$-vertex graph with more than $\frac{1}{2}(k-1)(n-1)$ edges, $k \ge 3$.
Then $G$ contains a cycle of length at least $k$.
\end{thm}

This result is best possible for $n - 1$ divisible by $k - 2$, due to any $n$-vertex graph where each block is a clique of order $k - 1$.
Let $\ex(n,P_k)$ be the maximum number of edges in an $n$-vertex graph with no $k$-vertex path; Theorem~\ref{ErdGallaiPath}
shows $\ex(n,P_k) \leq \frac{1}{2}(k - 2)n$ with equality for $n$ divisible by $k - 1$.  Several proofs and sharpenings of the Erd\H{o}s-Gallai theorem
were obtained by Woodall~\cite{Woodall}, Lewin~\cite{Lewin}, Faudree and Schelp\cite{FaudScheB,FaudSche75} and Kopylov~\cite{Kopy} -- see~\cite{FS224} for further details.
The strongest version was proved by
   Kopylov~\cite{Kopy}.
 To describe his result, we require
 the following graphs.  Suppose that $n \geq k$, 
    $(k/2) > a\geq 1$. Define the $n$-vertex graph $H_{n,k,a}$ as follows.
 The vertex set of $H_{n,k,a}$ is partitioned into three sets $A,B,C$ such that $|A| = a$, $|B| = n - k + a$ and $|C| = k - 2a$
 and the edge set of $H_{n,k,a}$ consists of all edges between $A$ and $B$ together with all edges in $A \cup C$.
Let
\[ h(n,k,a) : = e(H_{n,k,a}) = {k - a \choose 2} + a(n - k + a).\]

\begin{thm} [Kopylov  \cite{Kopy}] \label{th:Kopylov2}
Let $n \geq k \geq 5$ and $t = \lfloor \frac{k-1}{2}\rfloor$. If $G$ is an $n$-vertex 2-connected graph with no cycle of length at least $k$, then
\begin{equation}\label{eq:kop}
   e(G)\leq \max\left\{h(n,k,2), h(n,k,t)\right\}
\end{equation}
with equality only if $G = H_{n,k,2}$ or $G = H_{n,k,t}$.
 \end{thm}

In this paper, we prove a stability version of Theorems~\ref{ErdGallaiPath} and \ref{th:Kopylov2}. A {\em star forest} is a vertex-disjoint
union of stars.

\begin{thm}\label{t3}
Let $t \geq 2$ and $n \geq 3t$ and $k \in \{2t + 1,2t + 2\}$. Let $G$ be a 2-connected $n$-vertex graph containing no cycle of length at least $k$. Then
$e(G) \leq h(n,k,t-1)$ unless
\begin{center}
\begin{tabular}{lp{5.8in}}
$(a)$ & $k = 2t + 1$, $k \neq 7$, and $G \subseteq  H_{n,k,t}$ or \\
$(b)$ & $k = 2t + 2$ or $k = 7$, and $G - A$ is a star forest for some $A \subseteq V(G)$ of size at most $t$. \\
\end{tabular}
\end{center}
\end{thm}

This result is best possible in the following sense. Note that $H_{n,k,t - 1}$ contains no cycle of length at least $k$, is not a subgraph of $H_{n,k,t}$,
and $H_{n,2t+2,t - 1}-A$ has a cycle for every $A\subseteq V(H_{n,2t+2,t - 1})$ with $|A|=t$. Thus
the claim of Theorem \ref{t3} does not hold for $G = H_{n,k,t-1}$. Therefore the condition $e(G) \leq h(n,k,t-1)$ in Theorem \ref{t3} is best possible.
Since
$$ h(n,2t+2,t)=\binom{t}{2}+t(n-t)+1 =h(n,2t+1,t)+1$$
and
$$  h(n,2t+2,t-1)=\binom{t}{2}+(t-1)(n-t)+6 =h(n,2t+1,t-1)+3, $$
 the difference between Kopylov's bound and the bound in Theorem~\ref{t3} is
\begin{equation}\label{j13}
  h(n,k,t) - h(n,k,t-1) = \left\{\begin{array}{ll}
n - t - 3 & \mbox{ if }k = 2t + 1 \\
n - t - 5 & \mbox{ if }k = 2t + 2.
\end{array}\right.
\end{equation}

It is interesting that for a fixed $k$, the difference in~\eqref{j13} divided by $h(n,k,t)$ does not tend to $0$ when $n\to\infty$.

Theorem \ref{t3}
yields the following cleaner claim for $3$-connected graphs.

\begin{cor}\label{3con}
Let $k \geq 11$, $t = \lfloor \frac{k-1}{2}\rfloor$, and $n \geq \frac{3k}{2}$. If $G$ is an $n$-vertex $3$-connected graph with no cycle of length at least $k$,
then $e(G) \leq h(n,k,t-1)$ unless $G \subseteq  H_{n,k,t}$.
\end{cor}

In the same way that Theorem~\ref{ErdGallaiCyc} implies Theorem~\ref{ErdGallaiPath}, Theorem \ref{t3} applies to give a stability theorem for paths:

\begin{thm}\label{t4'} Let $t \geq 2$ and $n \geq 3t - 1$ and $k \in \{2t,2t+1\}$, and
let $G$ be a connected $n$-vertex graph containing no $k$-vertex path. Then $e(G) \leq h(n + 1,k + 1,t - 1) - n$ unless
\begin{center}
\begin{tabular}{lp{5.8in}}
$(a)$ $k = 2t$, $k \neq 6$, and $G \subseteq  H_{n ,k ,t-1}$ or \\
$(b)$ $k = 2t+1$ or $k = 6$, and $G - A$ is a star forest for some $A \subseteq V(G)$ of size at most $t-1$.
\end{tabular}
\end{center}
\end{thm}

Indeed, let $G'$ be obtained from an $n$-vertex connected graph $G$ with more than $h(n + 1,k+1,t-1) - n$ edges by adding a vertex adjacent to all vertices in $G$.
Then $G'$ is 2-connected and
$G'$ has more than $h(n+1,k+1,t-1)$ edges. If $G$ has no $k$-vertex path, then $G'$ has no cycle of length at least $k + 1$.
By Theorem \ref{t3}, $G'$ satisfies (a) or (b) in Theorem \ref{t3}, which means $G$ satisfies (a) or (b) in Theorem \ref{t4'}.
Repeating this argument, Corollary~\ref{3con} implies the following.

\begin{cor}\label{3conp}
Let $k \geq 11$, $t = \lfloor \frac{k-1}{2}\rfloor$, and $n \geq \frac{3k}{2}$. If $G$ is an $n$-vertex $2$-connected graph with no $k$-vertex paths,
then $e(G) \leq h(n+1,k+1,t-1)-n$ unless $G \subseteq  H_{n,k,t-1}$.
\end{cor}

\bigskip

 {\bf Organization.} The proof of Theorem \ref{t3}  will use a number of classical results listed in Section 2 and some lemmas on contractions proved in Section 3.
 Then in Section~\ref{maintheorem} we describe several families of extremal graphs and state and prove a more technical Theorem~\ref{main}, implying Theorem~\ref{t3}
 for $k\geq 9$.  Finally,
 in Section~\ref{secshort} we prove the analog of our technical  Theorem~\ref{main}
 for $4\leq k\leq 8$. In particular, we describe   {\em all} $2$-connected graphs with no cycles of length at least $6$.

 \bigskip

{\bf Notation.} We use standard notation of graph theory. Given a simple graph $G=(V,E)$, the {\em neighborhood} of $v\in V$, i.e. the set of vertices adjacent
to $v$, is denoted by $N_G(v)$ or $N(v)$ for short,
and the {\em closed neighborhood} is $N[v] := N(v) \cup \{v\}$. The {\em degree} of vertex $v$ is $d_G(v) := |N_G(v)|$.
Given $A\subseteq V$ we also use $N_G(v,A)$ for $N(v) \cap A$, $d(v,A)$ for $|N(v) \cap A|$, and $N(A) := \bigcup_{v\in A} N(v) \backslash A$. For an edge $xy$ in $G$, let $T_G(xy)$ denote the number of triangles containing $xy$
 and $T(G):= \min \{T_G(xy): xy\in E\}$. The minimum degree of $G$ is denoted by $\delta(G)$. For an edge $xy$ in $G$, $G/xy$ denotes the graph obtained from $G$ by contracting $xy$.
We frequently use $x* y$ for the new vertex. The length of the longest cycle in $G$ is denoted by $c(G)$, and $e(G) := |E|$.
Denote by $K_n$ the complete $n$-vertex graph, and $K(A,B)$ the complete bipartite graph with parts $A$ and $B$ ($A\cap B=\emptyset$).
Given vertex-disjoint graphs $G_1=(V_1,E_1)$ and $G_2=(V_2,E_2)$,  the graph $G_1 + G_2$ has vertex set $V_1\cup V_2$ and edge set $E_1\cup E_2\cup E(K(V_1,V_2))$.
If $G$ is a graph, then $\overline{G}$ denotes the complement of $G$ and for a positive integer $\ell$, $\ell G$ denotes the graph consisting of $\ell$ components, each isomorphic to $G$. For disjoint sets $A,B \subseteq V(G)$, let $G(A,B)$ denote the bipartite graph with parts $A$ and $B$ consisting of all edges of $G$ between $A$ and $B$,
and for $A \subseteq V(G)$, let $G[A]$ denote the subgraph induced by $A$.

\section{Classical theorems}

We require a number of theorems on long paths and cycles in dense graphs. The following is an extension to 2-connected
graphs of the well-known fact that an $n$-vertex non-hamiltonian graph has at most ${n - 1 \choose 2}+1$ edges:

\begin{thm}[Erd\H{o}s~\cite{Erd62}]\label{th:er}
Let $d\geq 1$ and $n>2d$ be integers, and
\[ \ell_{n,d}:=\max\left\{\binom{n-d}{2}+d^2,\binom{\lceil\frac{n+1}{2}\rceil}{2}+{\Big\lfloor\frac{n-1}{2}\Big\rfloor}^2\right\}.\]
Then every $n$-vertex graph $G$ with $\delta(G)\geq d$ and $e(G) > \ell_{n,d}$ is hamiltonian.
  \end{thm}

The bound on $\ell_{n,d}$ is sharp, due to the graphs $H_{n,n,2}$
and $H_{n,n,\lfloor (n-1)/2\rfloor}$. Since $\delta(G)\geq 2$ for every $2$-connected $G$, this has the following corollary.

\begin{thm}[Erd\H{o}s~\cite{Erd62}]\label{th:nonham}
If $n\geq 5$ and $G$ is an $n$-vertex $2$-connected non-hamiltonian graph, then
  $e(G)\leq \binom{n-2}{2}+4$, with equality only for $G=H_{n,n,2}$.
  \end{thm}

It is well-known that every graph of minimum degree at least $d \geq 2$ contains a cycle of length at least $d + 1$.
A stronger statement was proved by Dirac for 2-connected graphs:

\begin{thm}[Dirac~\cite{Dir}]\label{th:dirac}
If $G$ is $2$-connected then $c(G)\geq \min \{n,2\delta\}$.
  \end{thm}

This theorem was strengthened as follows by Kopylov~\cite{Kopy}, based on ideas of P\'{o}sa~\cite{Posa62}:

\begin{thm}[Kopylov~\cite{Kopy}]\label{le:kop}
If $G$ is $2$-connected, $P$ is an $x,y$-path of $\ell$ vertices, then $c(G)\geq \min \{\ell,d(x,P)+ d(y,P)\}$.
\end{thm}

\begin{thm}[Chv\' atal~\cite{Ch}]\label{t1} Let $n\geq 3$ and $G$ be an $n$-vertex graph with vertex degrees $d_1\leq d_2\leq\ldots\leq d_n$.
If $G$ is not hamiltonian, then there is some $i<n/2$ such that $d_i\leq i$ and $d_{n-i}<n-i$.
\end{thm}

The {\em $k$-closure} of a graph $G$ is the unique smallest graph $H$ of order $n:=|V(G)|$ such that $G\subseteq H$ and
 $d_H(u)+d_H(v)<k$ for all $uv\notin E(H)$. The $k$-closure of $G$ is denoted by $Cl_k(G)$, and can be obtained from $G$ by a recursive procedure which consists of joining nonadjacent vertices with degree-sum at least $k$.

\begin{thm}[Bondy and Chv\'atal~\cite{BonChv}]\label{th:hamclosure}
If $Cl_n(G)$ is hamiltonian, then so is $G$. Therefore if $Cl_n(G)=K_n$, $n\geq 3$, then $G$ is hamiltonian.
  \end{thm}

Concerning long paths between prescribed vertices in a graph, Lov\'{a}sz~\cite{Lovasz} showed that if $G$ is a 2-connected graph
in which every vertex other than $u$ and $v$ has degree at least $k$, then there is a $u,v$-path of length at least $k + 1$. This result was strengthened
by Enomoto.  The following theorem  immediately follows from  Corollary~1 in~\cite{En}:

\begin{thm}[Enomoto~\cite{En}]\label{t24} Let $5\leq s\leq n$ and $\ell:=2(n-3)/(s-4)$.
Suppose $H$ is a $3$-connected $n$-vertex graph with
$d(x)+d(y)\geq s$ for all non-adjacent distinct $x,y\in V(H)$. Then for every distinct vertices $x$ and $y$ of $H$, there is
an $x,y$-path  of length at
least $s-2$. Moreover, if for some distinct $x,y\in V(H)$, there is  no $x,y$-path  of length at least $s-1$, then either
\[ \overline{K_{s/2}}+\overline{K_{n-s/2}}\subseteq  H\subseteq  K_{s/2}+\overline{K_{n-s/2}} \]
  or $\ell$ is an integer and
\[ \overline{K_3}+\ell {K_{s/2-2}}\subseteq  H\subseteq  K_{3}+\ell {K_{s/2-2}}.\]
\end{thm}

A further strengthening of this result was given by Bondy and Jackson~\cite{BonJac}. Finally, we require some results on
cycles containing prescribed sets of edges. The following was proved by P\'{o}sa~\cite{Po}:

\begin{thm}[P\'osa~\cite{Po}]\label{t21} Let $n\geq 3$, $k<n$ and let $G$ be an $n$-vertex graph such that
\begin{equation}\label{m6}
d(u)+d(v)\geq n+k\qquad \mbox{ for every non-edge $uv$ in $G$.}
\end{equation}
Then for every linear forest $F$ with $k$ edges contained in $G$, the graph $G$
has a hamiltonian cycle containing all edges of $F$.
\end{thm}

The  analog of P\' osa's Theorem for bipartite graphs below is a simple corollary of Theorem~7.3 in~\cite{ZW}.

\begin{thm}[Zamani and West~\cite{ZW}]\label{zw}
Let $s\geq 3$ and  $K$ be a subgraph of the complete bipartite graph $K_{s,s}$ with
partite sets $A$ and $B$ such that for every $x\in A$ and $y\in B$ with $xy\notin E(K)$, $d(x)+d(y)\geq s+1+i$.
Then for every linear forest $F\subseteq K$ with at most $2i$ edges, there is a hamiltonian cycle in $K$ containing all edges of $F$.
\end{thm}
We will use only the following partial case of Theorem~\ref{zw}.

\begin{cor}\label{zw2}
 Let $s\geq 4$, $1\leq i\leq 2$ and $K$ be a subgraph of $K_{s,s}$ with at least $s^2-s+2+i$ edges.
If $F\subseteq K$ is a linear forest with at most $2i$ edges and at most two components,  then  $K$ has a hamiltonian cycle  containing all edges of $F$.
\end{cor}

\section{Lemmas on contractions}

An essential part of the proof of Theorem \ref{t3} is to analyze contractions of edges in graphs.
Specifically, we shall start with a graph $G$ and contract edges according to some basic rules.
Let us mention that the extensive use of contractions to prove the Erd\H{o}s--Gallai Theorem was
introduced by Lewin~\cite{Lewin}. In this section, we present some basic structural lemmas on contractions.

\begin{lem}\label{l2o} Let $n\geq 4$ and let $G$ be an  $n$-vertex $2$-connected graph.
Let $v\in V(G)$ and $W(v):=\{w\in N(v)\,:\, N[v]\not\subseteq  N[w]\}$. If $W(v)\neq \emptyset$,
then there is $w\in W(v)$ such that $G/vw$ is $2$-connected.
\end{lem}

{\bf Proof.} Let $w\in W(v)$, $G_w=G/vw$. Recall that $v*w$ is the vertex in $G_w$ obtained by contracting $v$ with $w$.
Since $G$ is $2$-connected, $G_w$ is connected. If $x\neq v*w$ is a cut vertex in $G_w$, then it is a cut vertex in $G$,
a contradiction. So, the only cut vertex in $G_w$ can be $v*w$. Thus, if the lemma does not hold, then for every $w\in W(v)$,
$v*w$ is the unique cut vertex in $G_w$. This means that  for every $w\in W(v)$, $\{v,w\}$ is a separating set in $G$.

Choose $w\in W(v)$ so that to minimize the order of a minimum component in $G-v-w$. Let $C$ be the vertex set of
such a component in $G-v-w$ and $C'=V(G) \setminus (C \cup \{v,w\})$. Since $G$ is $2$-connected, $v$ has a neighbor $u\in C$ and a
neighbor $u'\in C'$. Since $uu'\notin E(G)$, $u\in W(v)$. But the vertex set of every component of $G-v-u$ not
containing $w$ is contained in $C$. This contradicts the choice of $w$.\qed

This lemma yields the following fact.

\begin{lem}\label{l3o} Let $n\geq 4$ and let $G$ be an  $n$-vertex $2$-connected graph. For every $v \in V(G)$, there exists $w \in N(v)$ such that $G/vw$ is $2$-connected.
\end{lem}

{\bf Proof.} If $W(v)\neq \emptyset$, this follows from Lemma~\ref{l2o}. Suppose $W(v)= \emptyset$.
This means $G[N(v)]$ is a clique. Then contracting any edge incident with $v$ is equivalent to deleting $v$.
Let $G'=G-v$. Since $d(v)\geq 2$ and $G[N(v)]$ is a clique, any cut vertex in $G'$ is also a cut vertex in $G$.\qed

For an edge $xy$ in a graph $H$, let $T_H(xy)$ denote the number of triangles containing $xy$.  Let
$T(H):=\min\{T_H(xy)\,:\,xy\in E(H)\}$. When we contract an edge $uv$ in a graph $H$, the degree of every
$x\in V(H)\setminus\{u,v\}$ either does not change or decreases by $1$. Also the degree of $u*v$ in $H/uv$ is at least
$\max\{d_H(u),d_H(v)\}-1$. Thus
\begin{equation}\label{m254}
\mbox{$\delta(H/uv)\geq \delta(H)-1$ for every graph $H$ and $uv\in E(H)$.}
\end{equation}
Similarly,
\begin{equation}\label{m251}
\mbox{$T(H/uv)\geq T(H)-1$ for every graph $H$ and $uv\in E(H)$.}
\end{equation}

Suppose we contract edges of a $2$-connected graph one at a step, choosing always an edge $xy$ so that\\
(i) the new graph is $2$-connected and,\\
(ii) $xy$ is in the fewest triangles;\\
(iii) the contracted edge $xy$ is incident to a vertex of degree as small as possible up to (ii).

\begin{lem}\label{tl1}  Let $h$ be a positive integer. Suppose  a $2$-connected graph $G$  is obtained from a  $2$-connected graph $G'$ by contracting edge $xy$ into $x*y$
using the above rules (i)--(iii). If $G$ has at least $h$ vertices of degree at most $h$, then either $G'=K_{h+2}$ or $G'$ also has
a vertex of degree at most $h$.
\end{lem}

{\bf Proof.} Since $G$ is $2$-connected, $h\geq 2$.
If $G$ has a vertex of degree less than $h$, the lemma holds by~\eqref{m254}.
So, let $A_j$ denote the set of vertices of degree exactly $j$ in $G$, and assume $|A_h|\geq h$.
Let $A'_h = A_h \setminus \{x*y\}$.
Suppose the lemma does not hold.
Then we have
\begin{equation}\label{eq1'}
\mbox{
each $v\in A'_h$ has degree $h+1$ in $G'$ and is adjacent to both, $x$ and $y$ in $G'$.}
\end{equation}

{\bf Case 1:} $|A_h'|\geq h$. Then by~\eqref{eq1'}, $xy$ belongs to at least $h$ triangles in which
the third vertex is in $A_h$. So  by (iii) and the symmetry between $x$ and $y$, we may assume $d_{G'}(x)=h+1$.
This in turn yields $N_{G'}(x)=A_h \cup \{y\}$.
Since $G'$ is $2$-connected each
$v\in A_h'$ is not a cut vertex.
Even more, $xv$ is not a cut edge.
Indeed, $y$ is a common neighbor of all neighbors of $x$ so all neighbors of $x$ must be in the same
component as $y$ in $G' - x - v$. It follows that
\begin{equation}\label{eq:8}
\mbox{
for every $v\in A_h'$,  $G'/vx$ is $2$-connected.}
\end{equation}

If $uv\notin E(G)$ for some $u,v\in A_h$, then by~\eqref{eq:8} and (ii), we would contract the edge $xu$
and not $xy$. Thus $G'[A_h' \cup \{x,y\}]=K_{h+2}$ and so
either $G'=K_{h+2}$ or $y$ is a cut vertex in $G'$,
as claimed.

{\bf Case 2:} $|A_h'|=h-1$. Then $x*y\in A_h$. 
                                               We obtain that
 $d_{G'}(x)=d_{G'}(y)=h+1$
and $N_{G'}[x]=N_{G'}[y]$. So by~\eqref{eq1'}, there is $z\in V(G)$ such that $N_{G'}[x]=N_{G'}[y]=A_h'\cup \{x,y,z\}$.
Again~\eqref{eq:8} holds (for the same reason that $N_{G'}[x]\subseteq  N_{G'}[y]$).
Thus similarly $vu\in E(G')$ for every $v\in A_h'$ and every $u\in A_h' \cup \{z\}$. Hence
$G'[A_h'\cup \{x,y,z\}]=K_{h+2}$ and  either $G'=K_{h+2}$ or $z$ is a cut vertex in $G'$,
as claimed.\qed

\begin{lem}\label{tr3}
Suppose that $G$ is a $2$-connected graph and $C$ is a longest cycle in it.
Then no two consecutive vertices of $C$ form a separating set.
\end{lem}

\begin{proof}
Indeed, if for some $i$ the set $\{v_i,v_{i+1}\}$ is separating, then let $H_1$ and $H_2$ be two components of $G-\{v_i,v_{i+1}\}$
such that $V(C)\cap V(H_1)\neq \emptyset$.
Then $V(C)\setminus\{v_i,v_{i+1}\}\subseteq  V(H_1)$.
Let $x\in V(H_2)$.
Since $G$ is $2$-connected, it contains two paths from $x$ to $\{v_i,v_{i+1}\}$ that share only $x$.
Since $\{v_i,v_{i+1}\}$ separates $V(H_2)$ from the rest, these paths are fully contained in $V(H_2)\cup \{v_i,v_{i+1}\}$.
So adding these paths to $C-v_iv_{i+1}$ creates a cycle longer than $C$, a contradiction.
\end{proof}

\section{Proof of the main result, Theorem \ref{t3}, for $k \geq 9$}\label{maintheorem}

In this section, we give a precise description of the extremal graphs for Theorem \ref{t3} for $k\geq 9$.
The description for $k \leq 8$ is postponed to Section \ref{secshort}.
For Theorem \ref{t3}(a), when $k = 2t + 1$ and $t \neq 3$,
these are simply subgraphs of the graphs $H_{n,k,t}$: recall that $H_{n,k,a}$ has a partition into three sets $A,B,C$ such that $|A| = a$, $|B| = n - k + a$ and $|C| = k - 2a$
 and the edge set of $H_{n,k,a}$ consists of all edges between $A$ and $B$ together with all edges in $A \cup C$. For Theorem \ref{t3}(b), when $k = 2t + 2$ or $k = 7$,
 the extremal graphs $G$ contain a set $A$ of size at most $t$ such that $G - A$ is a star forest. In this case
 a more detailed description is required.

\medskip

{\bf Classes $\mathcal{G}_i(n,k)$ for $i \leq 3$.} Let ${\mathcal G}_1(n,k) := \{H_{n,k,t}\}$. Each $G\in {\mathcal G}_2(n,k)$ is defined by a partition $V(G)=A\cup B\cup J$, $|A| = t$ and a pair $a_1\in A$, $b_1\in B$
 such that $G[A]=K_t$, $G[B]$ is the empty graph, $G(A,B)$ is a complete bipartite graph and for every $c \in J$ one has $N(c)=\{a_1,b_1\}$. Every member of $G\in {\mathcal G}_3(n,k)$ is defined by a partition $V(G)=A\cup B\cup J$, $|A|=t$
 such that $G[A]=K_t$, $G(A,B)$ is a complete bipartite graph, and
 \begin{center}
 \begin{tabular}{lp{5in}}
$\bullet$ & $G[J]$ has more than one component \\
$\bullet$ & all components of $G[J]$ are stars with at least two vertices each \\
$\bullet$ & there is a $2$-element subset $A'$ of $A$ such that $N(J)\cap (A\cup B)=A'$ \\
$\bullet$ & for every component $S$ of $G[J]$ with at least $3$ vertices, all leaves of $S$ are adjacent to the same vertex $a(S)$ in $A'$.
\end{tabular}
\end{center}
The class $\mathcal{G}_4(n,k)$ is empty unless $k = 10$. Each member of ${\mathcal G}_4(n,10)$ has a 3-vertex set $A$ such that $G[A]=K_3$ and $G - A$ is a star forest such that
if a component $S$ of $G - A$ has more than two vertices then all its leaves are adjacent to the same vertex $a(S)$ in $A$.
These classes are illustrated below:
\vspace{-0.5in}
\begin{center}
\includegraphics[width=6.5in]{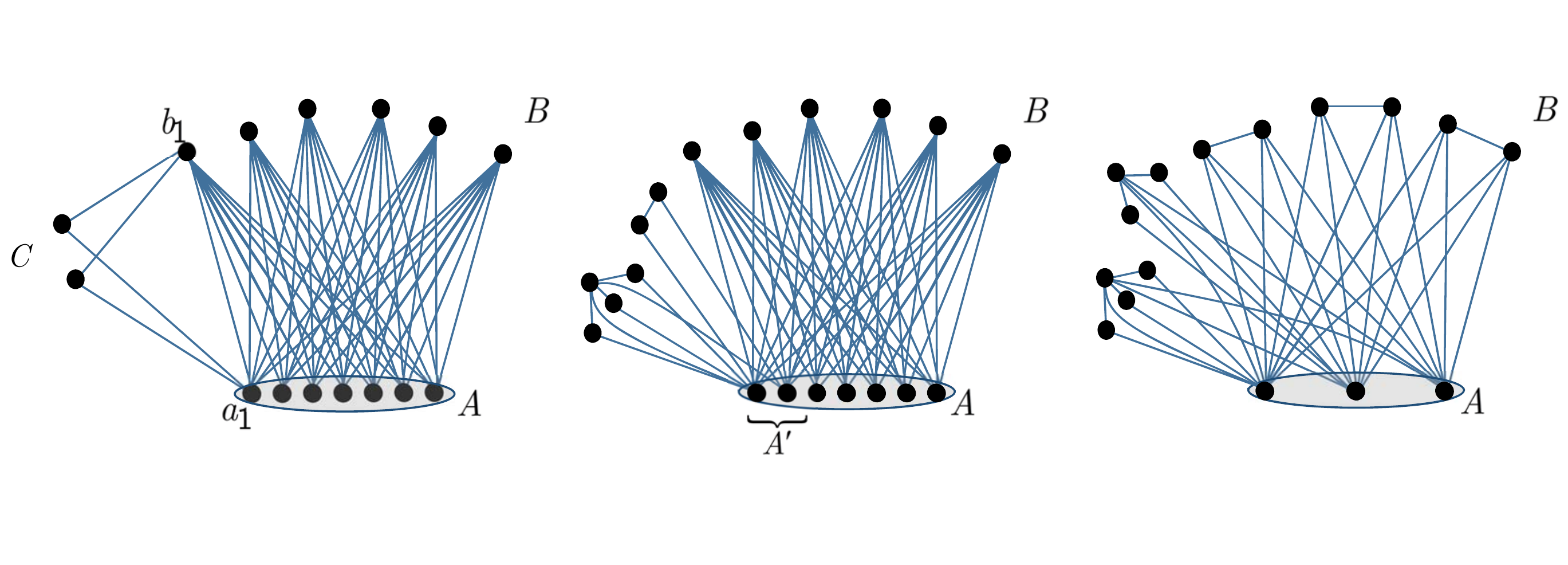}
\end{center}

\vspace{-0.7in}
\begin{center}
{\sf Figure 1: Classes $\mathcal{G}_2(n,k)$, $\mathcal{G}_3(n,k)$ and $\mathcal{G}_4(n,10)$.}
\end{center}

\medskip

{\bf Statement of main theorem.} Having defined the classes $\mathcal{G}_i(n,k)$ for $i \leq 4$, we now state a theorem
which implies Theorem \ref{t3} for $k \geq 9$ and shows that the extremal graphs are the graphs in the classes $\mathcal{G}_i(n,k)$:

\begin{thm}\label{main} {\rm (Main Theorem)}
Let $k \geq 9$, $n \geq \frac{3k}{2}$ and $t=\left\lfloor \frac{k-1}{2}\right\rfloor$. Let $G$ be an  $n$-vertex 2-connected graph  with no cycle of length at least $k$. Then
 $e(G)\leq h(n,k,t-1)$  or  $G$ is a subgraph of a graph in $\mathcal{G}(n,k)$, where
 \begin{center}
 \begin{tabular}{lp{5.8in}}
$(1)$ & if  $k$ is odd, then  $\mathcal{G}(n,k):= \mathcal{G}_1(n,k)=\{H_{n,k,t}\}$;\\
$(2)$ & if $k$ is even and $k \neq 10$, then  $\mathcal{G}(n,k):=\mathcal{G}_1(n,k) \cup \mathcal{G}_2(n,k) \cup \mathcal{G}_3(n,k)$; \\
$(3)$ & if  $k = 10$, then  $\mathcal{G}(n,k):=\mathcal{G}_1(n,10) \cup \mathcal{G}_2(n,10) \cup \mathcal{G}_3(n,10) \cup \mathcal{G}_4(n,10)$.
\end{tabular}
\end{center}
\end{thm}

We prove this theorem in  this section. We also observe that if $k \geq 11$, then the only graph in the classes $\mathcal{G}_i(n,k)$ that is 3-connected
is $H_{n,k,t}$. Therefore Theorem \ref{main} implies  Corollary~\ref{3con}.

\medskip

 The idea of the proof is to take a graph $G$ satisfying the conditions of the theorem with $c(G)<k$, and to contract edges while preserving the average degree and
 $2$-connectivity of $G$.
A key fact is that if a graph contains a cycle of length at least $k$ and is obtained from another graph by contracting edges,
 then that other graph also contains a cycle of length at least $k$.
The process terminates with an $m$-vertex graph $G_m$ such that $G_m$ is 2-connected, $m \geq k$, and if $m > k$ then $G_m$ has minimum degree at least $t - 1$.
If $m>k$, then we apply
Theorem \ref{t24} to show that $G_m$ is a dense subgraph of $H_{m,k,t}$. If $m = k$, then we apply Theorems \ref{th:er}, \ref{th:nonham}, \ref{t1},
and \ref{th:hamclosure} to show that $G_m$ is a dense
subgraph of $H_{k,k,t}$. Using this, we show that $G_m$ contains a dense nice subgraph.
 Analyzing  contractions, we then show  that $G$ itself contains a dense nice subgraph. Finally, we show that every dense $n$-vertex graph
containing a dense nice subgraph but not containing a cycle of length at least $k$
must be a subgraph of a graph in one of the classes
described in Theorem \ref{main}.

\subsection{Basic Procedure}

Let $k,n$ be positive integers with $n \geq k$. Let $G$ be an $n$-vertex $2$-connected graph with $c(G) <  k$ and $e(G) \geq h(n,k,t-1) + 1$. We denote $G$ as $G_n$ and run the following procedure.

\medskip

{\bf Basic Procedure.} At the beginning of each round, for some $j : k\leq j\leq n$, we have a $j$-vertex $2$-connected  graph $G_j$ with $e(G_j)\geq h(j,k,t-1) + 1$.
\begin{center}
\begin{tabular}{lp{5in}}
(R1) & If $j=k$, then we stop. \\
(R2) &  If there is an edge $xy$ with $T_{G_j}(xy)\leq t-2$ such that $G_j/xy$ is $2$-connected, choose one such edge so that\\
& (i) $T_{G_j}(xy)$ is minimum, and subject to this \\
& (ii) $xy$ is incident to a vertex of minimum possible degree.\\
& Then obtain $G_{j-1}$ by contracting $xy$. \\
(R3) & If (R2) does not hold, $j \geq k + t - 1$ and there is $uv \in E(G_j)$ such that $G_j - u - v$ has at least $3$ components and one of the components, say $H_1$ is a $K_{t-1}$,
then let $G_{j-t+1}=G_j-V(H_1)$. \\
(R4) & If neither (R2) nor (R3) occurs, then we stop.
\end{tabular}
\end{center}

{\bf Remark 1.} By construction, every obtained $G_j$ is $2$-connected and has $c(G_j) <  k$. Let us check that
\begin{equation}\label{0325}
 e(G_j) \geq h(j,k,t-1) + 1
 \end{equation}
  for all $m\leq j\leq n$.
For $j=n$,~\eqref{0325} holds by assumption. Suppose $j>m$ and~\eqref{0325} holds. If we apply (R2) to $G_j$, then
 the number of edges decreases by at most $t-1$, and $(h(j,k,t-1) + 1)-(h(j-1,k,t-1) + 1)=t-1$. If we apply (R2) to $G_j$, then
 the number of edges decreases by at most $\binom{t+1}{2}-1$, and $(h(j,k,t-1) + 1)-(h(j-(t-1)),k,t-1) + 1)=(t-1)^2$.
But for $k\geq 9$, $(t-1)^2\geq \binom{t+1}{2}-1$. Thus every step of the basic procedure preserves~\eqref{0325}.

\medskip

Let $G_m$ denote the graph with which the procedure terminates.
\medskip

{\bf Remark 2.} Note that if the rule (R3) applies for some $G_j$, then $\delta(G_j)\geq t$ and
the set $\{u,v\}$ is still separating in $G_{j-t+1}$, thus
$T_{G_{j-t+1}}(xy)\geq t-1$ for every edge $xy$ such that
$G_{j-t+1}/xy$ is $2$-connected. In particular, $\delta(G_{j-t+1})\geq t$.
So (R2) does not apply after any application of (R3) and $\delta(G_m)\geq t$.

\subsection{The structure of $G_m$}

In the next two subsections, we prove  Proposition~\ref{gmprop} below, considering the cases $m = k$ and $m > k$ separately.
Let $F_4$ be the graph obtained from $K_{3,6}$ by adding three independent edges in the part of size six.
In this section we usually suppose that $n\geq 3t$, $t\geq 4$, although many steps work for smaller values as well.

\begin{prop}\label{gmprop}
The graph $G_m$ satisfies the following properties:
\begin{center}
\begin{tabular}{lp{5in}}
$(1)$ & $G_m \subseteq  H_{m,k,t}$ or  \\
$(2)$ & $m > k = 10$ and $G_m \supseteq F_4$.
\end{tabular}
\end{center}
\end{prop}

\subsubsection{The case $m=k$}

If $G_k$ is hamiltonian, then $c(G)\geq k$,  a contradiction. So $G_k$ is not hamiltonian.

By Theorem~\ref{t1}, for every non-hamiltonian $n$-vertex graph $G$ with vertex degrees $d_1\leq d_2\leq\ldots\leq d_n$,
we define
\[ r(G):=\min \{i\;: d_i\leq i \,\mbox{\em and }\, d_{n-i}<n-i\}.\]

\begin{lem}\label{lem11} Let $t\geq 4$, $n\geq 3t$. If the vertex degrees  of  $G_k$ are $d_1\leq d_2\leq\ldots\leq d_{k}$,
 then  $r(G_k)=t$.
\end{lem}

{\bf Proof for $k = 2t + 2$.} Note that $r(G_k) \leq t$ since $r(G) < n/2$ (see Theorem \ref{t1}). Suppose $r:=r(G_k)\leq t-1$. Then by Remark~2, Rule (R3) never applied, and $G_k$ was obtained from $G$ by a sequence of $n-m$
edge contractions according (R2).
We may assume that for all $m\leq j<n$, graph $G_j$ was obtained from $G_{j+1}$ by contracting edge $x_jy_j$.
Then conditions for (R2) imply
\begin{equation}\label{03251}
T_{G_{j}}(x_{j-1}y_{j-1})\leq  t-2\quad\mbox{for every }\quad m+1\leq j\leq n.
\end{equation}

By Lemma~\ref{tl1}, $\delta(G_{m+1})\leq r$.
This together with~\eqref{03251} and~\eqref{m254} yield that for every $m < j \leq n$,
\begin{equation}\label{0324}
\mbox{ $\delta(G_j)\leq r+j-m-1$
 and so $T_{G_{j}}(x_{j-1}y_{j-1})\leq  \min\{r+j-m-2,t-2\}$.}
\end{equation}

 Contracting edge
 $x_{j-1}y_{j-1}$ in $G_j$, we lose $T_{G_{j}}(x_{j-1}y_{j-1})+1$ edges. Since $e(G) \geq h(n,k,t-1) + 1$, by~\eqref{m251} we obtain,
 \begin{eqnarray}\label{eqa1}
e(G_k) &\geq& h(n,k,t-1) + 1 - \sum_{j=m+1}^n\min\{t-1,r+j-m-1\} \\
&=& {t + 3 \choose 2} +(t - 1)(n - t - 3) +  1 - \sum_{j=m+1}^n\min\{t-1,r+j-m-1\} \nonumber \\
&=& {t + 3 \choose 2} +(t - 1)(n - t - 3)  + 1 - (t - 1)(n - m) + \sum_{j = m + 1}^n \max\{0,m+t - r - j\}  \nonumber \\
&=& \frac{3t^2+t+10}{2} + \sum_{j = m + 1}^n \max\{0,3t + 2 - r - j\}.  \nonumber
\end{eqnarray}
Since $n \geq 3t$, $\{\max\{0,3t+2-r-j\} : m + 1 \leq j \leq n\} = \{0,1,2,\dots,t-1-r\}$. Therefore
\begin{equation}\label{m8}
e(G_k) \geq \frac{3t^2+t+10}{2}+\sum_{i=1}^{t-1-r}i=\frac{3t^2+t+10}{2}+\binom{t-r}{2}.
\end{equation}
On the other hand, by the definition of $r$, $G_m$ has at most $r^2$ edges incident with the $r$ vertices of the smallest degrees and
at most $\binom{m-r}{2}$ other edges. Thus
  $e(G_m)\leq r^2+\binom{2t+2-r}{2}$.
Hence
\begin{equation}\label{a2}
\frac{3t^2+t+10}{2}+\binom{t-r}{2}\leq r^2+\binom{2t+2-r}{2}.
\end{equation}
Expanding the binomial terms in~\eqref{a2} and regrouping we get
\begin{equation}\label{m7}
t(r-3)\leq r^2-2r-4.
\end{equation}
If $r=3$, then the left hand side of~\eqref{m7} is $0$ and the right hand side is $-1$, a contradiction.
If $r\geq 4$, then dividing both sides of~\eqref{m7} by $r-3$ we get $t\leq r+1-  
  1/(r-3)$,  which yields $r\geq t$, as claimed.

So suppose $r=2$ and let $v_1,v_2$ be two vertices of degree $2$ in $G_k$.
Then by~\eqref{m8}, the graph $H=G_k-v_1-v_2$
has at least $$\frac{3t^2+t+10}{2}+\binom{t-2}{2}-2(2)=2t^2-2t+4$$ edges. So the complement of $H$ has at most $t-4$ edges and thus, for $u,w \in V(H)$:
\[ d_H(u)+d_H(w)\geq 2(2t-1)-(t-4)-1=3t+1=|V(H)|+t+1.\]
Hence by Theorem~\ref{t21},
\begin{equation}\label{m9}
\mbox{\em for each linear forest $F \subseteq H$ with $e(F) \leq t+1$, $H$
has a spanning cycle containing $E(F)$.}
\end{equation}
If $N(v_i)=\{u_i,w_i\}$ for $i=1,2$ and $v_1v_2\in E(G_k)$, say $u_1=v_2$ and $u_2=v_1$,
then by~\eqref{m9}, graph $H'=H+w_1w_2$ has a spanning cycle containing $w_1w_2$, and this cycle yields a hamiltonian cycle
in $G_k$, a contradiction. So $v_1v_2\notin E(G_k)$. Similarly, if
$N(v_1)\neq N(v_2)$, then by~\eqref{m9}, graph $H''=H+u_1w_1+u_2w_2$ has a spanning cycle containing $u_1w_1$ and $u_2w_2$. Note $w_1 \neq w_2$ since $H$ is 2-connected.
Again this yields a hamiltonian cycle
in $G_k$. Thus we may assume $N(v_1)=N(v_2)=\{u,w\}$. Let
\begin{equation}\label{0323}
\mbox{$H_0=H+uw$ if $uw\notin E(G)$ and $H_0=H$ otherwise.}
\end{equation}

If $x_m*y_m\notin N[v_1]\cup N[v_2]$, then $T_{G_{m+1}}(x_{m}y_{m})\leq 1$ (since  $T_{G_{m+1}}(v_1u_1)\leq 1$)
and  $G_{m+1}$ contains vertices $v_1$ and $v_2$ of degree $2$. So by Lemma~\ref{tl1} for $h=2$, $G_{m+2}$ also has
a vertex of degree~$2$. Thus by~\eqref{m254} for $r=2$
 instead of~\eqref{0324} we have
for every $m+2\leq j \leq n$,
\begin{equation}\label{03241}
\mbox{ $\delta(G_j)\leq \min\{j-m,t-1\}$
 and so $T_{G_{j}}(x_{j-1}y_{j-1})\leq  \min\{j-m-1,t-2\}$.}
\end{equation}
Plugging~\eqref{03241} instead of~\eqref{0324} into~\eqref{eqa1} for $r=2$, we
will instead of~\eqref{a2} get the stronger inequality
\begin{equation}\label{a21}
\frac{3t^2+t+10}{2}+(t-3)+\binom{t-2}{2}\leq 2^2+\binom{2t+2-2}{2}.
\end{equation}
Thus instead of~\eqref{m7} we have for $r=2$ the stronger
inequality
$t(2-3)+(t-3)\leq 2^2-4-4$,
 which does not hold. This contradiction implies $x_{m}*y_{m}\in N[v_1]\cup N[v_2]$.
By symmetry we have two cases.

{\bf Case 1:} $x_{m}*y_{m}=v_1$. As above, graph $H_0$ has a spanning cycle $C$ containing $uw$. If
\begin{equation}\label{m10}
x_{m}u,y_{m}w\in E(G_{m+1}),
\end{equation}
then $C$ extends to a $k$-cycle in $G_{m+1}$ by replacing $uw$ with path $u,x_{m},y_{m},w$. A similar situation holds if
\begin{equation}\label{m11}
x_{m}w,y_{m}u\in E(G_{m+1}).
\end{equation}
 But by degree conditions each of $x_{m},y_{m}$ has a neighbor in $\{u,w\}$. By definition, each of
$u,w$ has a neighbor in $\{x_{m},y_{m}\}$. So at least one of~\eqref{m10} and~\eqref{m11} holds.

{\bf Case 2:} $x_{m}*y_{m}=u$. If $d_{G_{m+1}}(v_1)=d_{G_{m+1}}(v_2)=2$, then as before we get~\eqref{a21}
instead of~\eqref{m7} and get a contradiction. So by symmetry we may assume that $v_1$ is adjacent to
both $x_{m}$ and $y_{m}$ in $G_{m+1}$. Since $G_m$ is $2$-connected, vertex $w$ does not separate
 $\{v_1, v_2,u\}$ from the rest of the graph. Thus by symmetry we may assume that
$y_{m}$  has a neighbor $z\in V(G_{m+1})\setminus\{x_{m}, v_1,v_2,w\}$.
Again
by~\eqref{m9}, graph $H_0$ defined by~\eqref{0323} has  a spanning cycle containing edges $uw$ and $uz$, and again this cycle yields a
$k$-cycle in $G_{m+1}$ (using path $w,v_1,x_{m},y_{m},z$), a contradiction.


\medskip

{\bf Proof for $k = 2t + 1$.}  We repeat the argument for $k = 2t + 2$, but instead of~\eqref{m8} and \eqref{a2}, we get
\begin{equation*}
\frac{3t^2-t+6}{2}+\binom{t-r}{2}\leq e(G_k)\leq r^2+\binom{2t+1-r}{2}.
\end{equation*}
Expanding the binomial terms and regrouping, similarly to~\eqref{m7}, we get
\begin{equation*}
t(r-2)\leq r^2-r-3.
\end{equation*}
The analysis of this inequality is simpler than that of~\eqref{m7}:
If $r=2$, then the left hand side  is $0$ and the right hand side is $-1$, while if $r\geq 3$, then dividing both sides
   by $r-2$ we get $t\leq r+1-  
   1/(r-2)$, which yields $r\geq t$, as claimed. \qed

\begin{lem}\label{lem1} Under the conditions of   Lemma~\ref{lem11},
 $G_k$ is a subgraph of the graph $H_{k,k,t}$.
\end{lem}

{\bf Proof for $k = 2t + 2$.}   By Lemma~\ref{lem11}, $r(G_k)=t$.
Let $G'$ be the $k$-closure of $G_k$ and $d'_1\leq d'_2\leq\ldots\leq d'_{k}$ be the vertex degrees in $G'$.
By the definition of the $k$-closure,
\begin{equation}\label{eq3}
d(u)+d(v)\leq k-1 \qquad \mbox{\em for every non-edge $uv$ in $G'$.}
\end{equation}

Since $d'_i\geq d_i$ for every $i$ and $G'$ is also non-hamiltonian, $r(G')\geq r(G_k)=t$. Since $r(G') \leq t$ from $r(G) < n/2$, $r(G')=t$.
Let $V(G')=\{v_1,\ldots,v_{k}\}$ where $d_{G'}(v_i)=d'_i$ for all $i$. By the definition of $r(G')$,  on the one hand $d'_t\leq t$ and $d'_{k-t}\leq k-t-1=t+1$,
on the other hand
either $d'_{t-1}>t-1$ or  $d'_{k-(t-1)}\geq k-(t-1)=t+3$.
In any case, $d'_{t+3}\geq t$.
Summarizing,
\begin{equation}\label{eq2}
d'_{t+3}\geq t,\; d'_t\leq t\; \mbox{\em and   }\;
d'_{t+1}\leq d'_{t+2}\leq t+1.
\end{equation}
Let $B=\{v_1,\ldots,v_{t+2}\}$ and $A=V(G')\setminus B$.
If $d'_{t+4}\leq t+2$, then $$\sum_{i=1}^{k}d'_i\leq (t|B|+2)+(t+2)2+(2t+1)(t-2)=3t^2+t+4,$$ a contradiction to $e(G_k) \geq h(k,k,t-1) + 1$. Thus
$d'_{t + 4} \geq t + 3$, and by~\eqref{eq3} and~\eqref{eq2}, $G'[A]=K_{t}$. In summary,
\begin{equation}\label{eq4}
d'_{t+4}\geq t+3 \quad \mbox{\em and } \quad G'[A]=K_{t}.
\end{equation}

Suppose that there are distinct $v_{i_1},v_{i_2}\in B$ and distinct $v_{j_1},v_{j_2}\in A$ such that  $v_{i_1}v_{j_1}$ and
$v_{i_2}v_{j_2}$ are non-edges in $G'$. Then by~\eqref{eq3} and~\eqref{eq2},
\begin{eqnarray*}
\sum_{i=1}^{2t+2}d'_i &\leq& (2t+1)2+t(|B|-2)+2+(2t+1)(|A|-2)\\
&=& 4t+2+t^2+2+2t^2-3t-2 \; \; = \; \; 3t^2+t+2.
\end{eqnarray*}
This contradicts $e(G_k) > h(k,k,t-1)$. So, some $v_j$ is incident with all non-edges of $G'$ connecting  $A$ with $B$.

 {\bf Case 1:} $j\leq t+2$, i.e. $v_j\in B$. Then each $v\in B-v_j$ has $t$ neighbors in $A$.
Thus each $v\in B \setminus \{v_j,v_{t+1},v_{t+2}\}$ has no neighbors
 in $B$, and each of $v_{t+1}, v_{t+2}$ has  at most one neighbor in $B$. If each of $v_{t+1}, v_{t+2}$ is adjacent to $v_j$,
 then $G'$ has a hamiltonian cycle using edges $v_{t+1}v_j$  and $v_j v_{t+2}$. Otherwise $G'[B]$ has at most one edge, as claimed.

  {\bf Case 2:} $j\geq t+3$, i.e. $v_j\in A$. Together with~\eqref{eq4}, this yields that
   $G'$ contains $K_{t-1,t+3}$ with partite sets $A \setminus\{v_j\}$ and $B \cup \{v_j\}$. In particular, all pairs of vertices in $A \setminus\{v_j\}$ are adjacent. So, $G'$ is obtained from $K_{2t+2}-E(K_{t+3})$ by adding at least $e(G') - {2t + 2 \choose 2} + {t + 3 \choose 2} \geq 7$ edges.
If $G'[B \cup \{v_j\}]$ contains a linear forest with four edges,
 then $G'$ has a hamiltonian cycle. So suppose
   \begin{equation}\label{eq5}
\mbox{\em $G'[B \cup \{v_j\}]$ contains no linear forests with four edges, }
\end{equation}

{\bf Case 2.1:}  $G'[B \cup \{v_j\}]$ contains a cycle $C$. By~\eqref{eq5}, $|C|\leq 4$ and if $|C|=4$, then each other
edge in $G'[B \cup \{v_j\}]$ has both ends in $V(C)$. Thus $G'[B \cup \{v_j\}]$ has at most $6$ edges, a contradiction.
So suppose $C=(x,y,z)$. If no other edge is incident with $V(C)$, then the set of the remaining at least four edges in $G'[B \cup \{v_j\}]$
contains a linear forest with two edges, a contradiction to~\eqref{eq5}. Thus we may assume that $G'[B \cup \{v_j\}]$
has an edge $xu$ where $u\notin \{y,z\}$. Then by~\eqref{eq5} and the fact that  $G'[B \cup \{v_j\}]$ contains no $4$-cycles, none of $u,y,z$ is incident with other edges.
On the other hand, if $G'[B \cup \{v_j\}]$ has an edge
not incident with $V(C)$, this would contradict~\eqref{eq5}. Hence $G'[B \cup \{v_j\} \setminus\{x\}]$ has only the edge $yz$, as claimed.

  {\bf Case 2.2:}  $G'[B \cup \{v_j\}]$ is a forest.  By~\eqref{eq5},
  there is $x\in B \cup \{v_j\}$ of degree at least $3$  in  $G'[B \cup \{v_j\}]$. If there is another vertex $y$ of degree at least $3$ in $G'[B \cup \{v_j\}]$,
 then we can choose two edges incident with $x$ and two edges incident with $y$ that together form a linear forest with four edges.
So $G'[B \cup \{v_j\}\setminus\{x\}]$ is a linear forest, call it $F$, and thus has at most $3$ edges. Each edge of $F$ has at most one end adjacent
to $x$ and the degree of $x$ in $G'[B \cup \{v_j\}]$ is at least four. So if $F$ has exactly $m\in \{2,3\}$ edges, then
we can choose $4-m$ edges incident with $x$ so that together with $F$ they form a linear forest. And if $F$  has at most
one edge, then the lemma holds.

\medskip

{\bf Proof for $k = 2t + 1$.}  The proof is almost identical to the case $k = 2t + 2$. By Lemma~\ref{lem11}, $r(G_k)=t$.
Let $G'$ be the $k$-closure of $G_k$ and $d'_1\leq d'_2\leq\ldots\leq d'_k$ be the vertex degrees in $G'$.
As in~\eqref{eq3}, we have
\begin{equation}\label{eq3'}
d(u)+d(v)\leq k - 1 = 2t \qquad \mbox{\em for every non-edge $uv$ in $G'$.}
\end{equation}

As in the proof in the case $k = 2t + 2$,  $r(G')=t$.
Let $V(G')=\{v_1,\ldots,v_k\}$ where $d_{G'}(v_i)=d'_i$ for all $i$. Instead of~\eqref{eq2}, we get
the stronger claim
\begin{equation}\label{eq2'}
d'_{t+2}\geq t\; \mbox{\em and   }\;d'_t\leq d'_{t+1}=t.
\end{equation}
Let $B=\{v_1,\ldots,v_{t+1}\}$ and $A=V(G')\setminus B$.
If $d'_{t+3}\leq t+1$, then $$\sum_{i=1}^{2t+1}d'_i\leq t|B|+(t+1)2+(2t)(t-2)=3t^2-t+2 \leq h(k,k,t-1),$$ a contradiction. Thus,
\begin{equation}\label{eq4'}
d'_{t+3}\geq t+2\; \mbox{  so by~\eqref{eq3'} and~\eqref{eq2'}, }\;
G'[A]=K_{t}.
\end{equation}

If there are distinct $v_{i_1},v_{i_2}\in B$ and distinct $v_{j_1},v_{j_2}\in A$ such that  $v_{i_1}v_{j_1}$ and
$v_{i_2}v_{j_2}$ are non-edges in $G'$, then by~\eqref{eq3'} and~\eqref{eq2'},
$$\sum_{i=1}^k d'_i\leq (2t)2+t(|B|-2)+(2t)(|A|-2)=4t+t^2-t+2t^2-4t=3t^2-t \leq h(k,k,t-1),$$
 a contradiction. So, some $v_j$ is incident with all non-edges of $G'$ connecting  $A$ with $B$.

 {\bf Case 1:} $j\leq t+1$, i.e. $v_j\in B$. Then each $v\in B-v_j$ has $t$ neighbors in $A$.  Thus by~\eqref{eq2'},
 each $v\in B-v_j$
 has no neighbors
 in $B$, hence $B$ is independent, as claimed.

  {\bf Case 2:} $j\geq t+2$, i.e. $v_j\in A$. Together with~\eqref{eq4'}, this yields that
   $G'-v_j$ contains $K_{t-1,t+2}$ with partite sets $A \setminus\{v_j\}$ and $B \cup \{v_j\}$. In particular, each vertex in $A \setminus\{v_j\}$ is
   all-adjacent. So, $G'$ is obtained from $K_k - E(K_{t+2})$ by adding at least four edges.
If $G'[B \cup \{v_j\}]$ contains a linear forest with three edges,
 then $G'$ has a hamiltonian cycle. Every graph with at least four edges not containing a linear forest with three edges is a star plus isolated vertices.
 And if $G'[B \cup \{v_j\}]$ is a star plus isolated vertices, then
  $G'\subseteq H_{k,k,t}$.  \qed

\subsubsection{The case $m > k$.}

\begin{lem}\label{nlem} Let $m > k \geq 9$.
\begin{center}
\begin{tabular}{lp{5in}}
$(1)$ & If $k \neq 10$, then $G_m \subseteq  H_{m,k,t}$. \\
$(2)$ & If $k = 10$ then $G_m \subseteq  H_{m,k,t}$ or $G_m \supseteq F_4$.
\end{tabular}
\end{center}
\end{lem}

{\bf Proof for $k = 2t + 2$.}
$G_m$ is an $m$-vertex $2$-connected graph with  $c(G_m)\leq 2t+1$
satisfying $e(G) \geq h(n,k,t-1) + 1$.
Since (R2) is not applicable,
\begin{equation}\label{tr1}
\mbox{\em $T_{G_m}(xy)\geq t-1$ for every non-separating edge $xy$. }
\end{equation}
By Lemmas~\ref{l3o} and~\ref{l2o},~\eqref{tr1} implies
\begin{equation}\label{tr2}
\mbox{\em $\delta(G_m)\geq t$ and for each $v\in V(G_m)$ with $d(v)=t$, $G_m[N(v)]=K_{t+1}$. }
\end{equation}

Let $C=(v_1,\ldots,v_q)$ be a longest cycle in $G_m$.
Since $\delta(G_m)\geq t$, Dirac's Theorem (Theorem~\ref{th:dirac}) yields $q\geq 2t$.
Obviously, $q\leq 2t+1$.

By~\eqref{tr1} and Lemma~\ref{tr3}, each edge of $C$ is in at least $t-1$ triangles. By the maximality of $C$, the third vertex of each such triangle
is in $V(C)$. So
\begin{equation}\label{tr4}
\mbox{\em the minimum degree of $G_m[V(C)]$ is at least $t$.}
\end{equation}

We now prove that
\begin{equation}\label{tr5}
\mbox{\em $G_m[V(C)]$ is $3$-connected.}
\end{equation}

Indeed, assume~\eqref{tr5} fails and $G_m[V(C)]$ has a separating set $S$ of size $2$.
By symmetry, we may assume that  $S=\{v_1,v_j\}$ and that $j\leq \lfloor q/2\rfloor+1\leq t+1$.
Then by~\eqref{tr4}, $j=t+1$ and $G_m[\{v_1,\ldots,v_{t+1}\}]=K_{t+1}$. In particular,
\begin{equation}\label{m25}
v_1v_{t+1}\in E(G_m).
\end{equation}
Let $H_1=G_m[\{v_1,\ldots,v_{t+1}\}]$ and $H_2=G_m[\{v_{t+1},\ldots,v_{q},v_1\}]$. Similarly to $H_1$, graph $H_2$ is
either $K_{t+1}$ (when $q=2t$) or is obtained from $K_{t+2}$ by deleting some matching (when $q=2t+1$).

Concerning almost complete graphs we need the following statement which is an easy consequence of Theorem~\ref{t21}
(or one can prove it directly).

\begin{equation}\label{tr6a}
\parbox{4.5in}{\em For $p \geq 6$ and for any matching $M \subseteq K_p$, every two edges of $K_p - M$ are in a common hamiltonian cycle of $K_p - M$.}
\end{equation}

Since $G_m$ is $2$-connected, each component $F$ of $G_m-V(C)$ has at least two neighbors, say $y(F)$ and $y'(F)$, in $C$.
If at least one of them, say $y'(F)$, is not in $S=\{v_1,v_{t+1}\}$, then we can construct a cycle longer than $C$ as follows.

If $y(F)\in V(H_1)\setminus\{v_1,v_{t+1}\}$ and
$y'(F)\in V(H_2)\setminus\{v_1,v_{t+1}\}$, then $H_1-v_{t+1}$ has a hamiltonian $v_1,y(F)$-path $P_1$ (recall that $H_1-v_{t+1}$ is a complete graph),
and $H_2$ has a hamiltonian $v_1,y'(F)$-path $P_2$, by (\ref{tr6a}) and since $k \geq 4$.
So $P_1 \cup P_2$ and a $y(F),y'(F)$-path through $F$ form a longer than $C$
cycle in $G_m$.

If both, $y(F)$ and $y'(F)$ are in the same $H_j$, then we let $H'_j$ be the graph obtained from $H_j$ by adding the
 edge $y(F)y'(F)$. Recall that by~\eqref{m25}, $v_1v_{t+1}\in E(H_j)$.
If we have a hamiltonian cycle $C'$ in $H'_j$ containing $y(F)y'(F)$ and $v_1v_{t+1}$,
 then let $P$ be the  $v_1,v_{t+1}$-path obtained from $C'$ by
deleting edge $v_1v_{t+1}$ and  replacing edge $y(F)y'(F)$  with a $y(F),y'(F)$-path $P'$ through $F$, and then replace in $C$ the
$v_1,v_{t+1}$-path through $V(H_j)$ with the longer path $P$.
There is such a $C'$ if $|V(H_j)|\geq 6$ by \eqref{tr6a}, and also if $|V(H_j)|=5$ because in the latter case
$|V(H_j)|=t+1$ with $t=4$ and it is a complete graph.

Thus every component $F$ of $G_m-V(C)$ is adjacent only to $S$, and $S$ is a separating set in $G_m$.
In particular, $H_1-S=K_{t-1}$ and $H_2-S$ are components of $G_m-S$.
So, if $m\geq 3t+1$, then Rule (R3) is applicable, contradicting the definition of $G_m$. Hence $2t+2\leq m\leq 3t$.
On the other hand, by~\eqref{tr2}, every component of $G_m-S$ has at least $t-1$ vertices, and so
 $m-q\geq t-1$. Therefore,  $3t-1\leq m\leq 3t$.

If $m=3t-1$,
 then $q=2t$, $H_2=K_{t+1}$ and $H_3:=G_m-(V(C)-S)=K_{t+1}$. Hence
\begin{multline*}
   e(G_m)-h(m,k,t-1)-1=3\binom{t+1}{2}-2-h(3t-1,k,t-1)-1\\
    =\frac{3t^2+3t-4}{2}-\frac{5t^2-7t+16}{2}=-t^2+5t-10<0.
  \end{multline*}
Similarly, if $m=3t$,
 then the component sizes of $G_m-S$ are $t,t-1,t-1$. Thus in this case
\begin{multline*}
   e(G_m)-h(m,k,t-1)-1\leq t^2+t+\binom{t+2}{2}-2-h(3t,k,t-1)-1\\
   =\frac{3t^2+5t}{2}-1-\frac{5t^2-5t+14}{2}=-t^2+5t-8<0.
  \end{multline*}
These contradictions prove~\eqref{tr5}.

So by~\eqref{tr5} and Theorem~\ref{t24} for $n=q$, $s=2t$ and $H=G_m[V(C)]$, one of three cases below holds:\smallskip

{\bf Case 1:}  $  \overline{K_{t}}+\overline{K_{q-t}}\subseteq  G_m[V(C)]\subseteq  K_{t}+\overline{K_{q-t}}$.
Let $B$ be the independent set of size $q-t$ in $G_m[V(C)]$ and $A=V(C)\setminus B$.
In this case, since $G_m[V(C)]$ has hamiltonian cycle $C$ and an independent set $B$ of size $q-t$, we need $q=2t$.

Suppose that $G_m-V(C)$ has a component  $D$ with at least two vertices. By Menger's Theorem, there are two fully disjoint paths,
say $P_1$ and $P_2$,
connecting some two distinct vertices, say $u$ and $v$, of $D$ with two distinct vertices, say $x$ and $y$, of $C$. Since
$G_m[V(C)]$ contains $K_{t,t}$, it has
 an $x,y$-path with at least $2t-1$ vertices. This path together with $P_1, P_2$ and a $u,v$-path in $D$ form
a cycle of length at least $2t+1$, a contradiction to the maximality of $C$.
Thus each component of $G_m-V(C)$ is a single vertex and is adjacent either only to
vertices in $A$ or only to vertices in $B$. Moreover, by~\eqref{tr2}, each such vertex  has degree exactly $t$, and thus
its neighborhood is a complete graph. Since $B$ is independent, each $v\in V(G_m)-C$ is adjacent only to vertices in $A$.
Thus $G_m=K_m-E(K_{m-t})=H_{m,k-1,t}\subseteq  H_{m,k,t}$.\smallskip

{\bf Case 2:}  $ \overline{K_{3}}+\ell {K_{t-2}}\subseteq  G_m[V(C)] \subseteq  K_{3}+\ell {K_{t-2}},
$
where $\ell= 2(q-3)/(2t-4)$.    
Again, since $G_m[V(C)]$ has hamiltonian cycle $C$ and a separating set of size $3$
(call this set $A$), $\ell\leq 3$. If $\ell\leq 2$, then $q\leq 3+2(t-2)<2t$, a contradiction. Thus,
$\ell=3$ and $q=3+3(t-2)=3t-3$. Since $2t\leq q\leq 2t+1$, we get $t\in \{3,4\}$.
Since $t \geq 4$ by assumption, we obtain that $t=4$ and $F_4 \subseteq  G_m$. \smallskip

{\bf Case 3:}  {\em For every two distinct $x,y\in V(C)$, the graph  $G_m[V(C)]$ contains an $x,y$-path with at least $2t$ vertices.}
Let $W=V(G_m)-V(C)$. Repeating the argument of the second paragraph of  Case 1, we obtain that in our case
\begin{equation}\label{03253}
\mbox{\em each component of $G_m[W]$ is a singleton and so $N(w)\subseteq  V(C)$ for each $w\in W$.}
\end{equation}
Since no $w\in W$ is adjacent to two consecutive vertices of $C$ (by the maximality of $C$) and $q\leq 2t+1$, by~\eqref{tr2},
\begin{equation}\label{a12}
\mbox{\em $d_{G_m}(w)=t$ for every $w\in W$. }
\end{equation}
Fix some $w_1\in W$. Then we may relabel the vertices of $C$ so that
$N_{G_m}(w_1)=\{v_1,v_3,v_5,\ldots,v_{2t-1}\}$. By~\eqref{tr2}, this also yields $G_m[\{v_1,v_3,\ldots,v_{2t-1}\}]=K_t$
and thus $d_{G_m}(v_i)\geq t+1$ for all $i\in \{1,3,\ldots,2t-1\}$. In particular,
\begin{equation}\label{a13}
\mbox{\em $d_{G_m}(v)\geq t+1$ for every $v\in N_{G_m}(w_1)$. }
\end{equation}
 Then for every $j\in \{2,4,\ldots,2t-2\}$ (and for $j=2t$ in the case $q=2t$)
we can replace $v_j$ with $w_1$ in $C$ and obtain another longest cycle. By~\eqref{a12} and~\eqref{03253}, this yields $d_{G_m}(v_j)=t$
and
\begin{equation}\label{03254}
\mbox{\em
 $N_{G_m}(v_j)\subseteq  V(C)$ for all  $j\in \{2,4,\ldots,2t-2\}$ (and for $j=2t$ in the case $q=2t$).}
\end{equation}

{\bf Case 3.1:} $q=2t$. Switching the roles of $w_1$ with $v_j$ together with~\eqref{a13} yields
\begin{equation}\label{a14}
\mbox{\em $N_{G_m}(v_j)=\{v_1,v_3,v_5,\ldots,v_{2t-1}\}$ for all $j=2,4,\ldots,2t$.}
\end{equation}
 By~\eqref{a12} and~\eqref{a14},
$N_{G_m}(w)=\{v_1,v_3,v_5,\ldots,v_{2t-1}\}$ for all $w\in V(G_m)-\{v_1,v_3,v_5,\ldots,v_{2t-1}\}$. This means
$G_m\subseteq  H_{m,2t+2,t}$, as claimed.

{\bf Case 3.2:} $q=2t+1$. Since $m\geq 2t+3$, there is $w_2\in W-w_1$. By~\eqref{03254}, vertex $w_2$
is not adjacent to $v_j$ for $j\in \{2,4,\ldots,2t-2\}$. Suppose that $w_2$ is adjacent to $v_{2t}$ or $v_{2t+1}$, say $w_2v_{2t}\in E(G_m)$.
Then by the maximality of $C$, $w_2v_{2t+1},w_2v_{2t-1}\notin E(G_m)$. So the only possible $t$-element set of neighbors of $w_2$
is $\{v_1,v_3,\ldots,v_{2t-3},v_{2t}\}$. But then $G_m$ has the $(2t+2)$-cycle $(w_2,v_3,v_4,v_5,\ldots,v_{2t-1},w_1,v_1,v_{2t+1},v_{2t},w_2)$,
a contradiction. Thus
\begin{equation}\label{a15}
\mbox{\em
$N_{G_m}(w)=\{v_1,v_3,v_5,\ldots,v_{2t-1}\}$ for all $w\in W$.}
\end{equation}
Since we can replace in $C$ any $v_j$ for $j\in \{2,4,\ldots,2t-2\}$ with $w_1$,~\eqref{a15} yields
$N_{G_m}(v_j)=\{v_1,v_3,v_5,\ldots,v_{2t-1}\}$ for all $j=2,4,\ldots,2t-2$. It follows that
$\{v_1,v_3,v_5,\ldots,v_{2t-1}\}$ covers all edges in $G_m$ apart from edge $v_{2t}v_{2t+1}$.
This means $G_m\subseteq  H_{m,2t+2,t}$, as claimed.

\medskip

{\bf Proof for $k = 2t + 1$.} Similarly to the proof for $k = 2t + 2$, we have~\eqref{tr1} and~\eqref{tr2}.
Let $C=(v_1,\ldots,v_q)$ be a longest cycle in $G_m$.
Since $\delta(G_m)\geq t$, by Theorem \ref{th:dirac}, $q\geq 2t$; so
  $c(G_m) < k$ yields $q = 2t$. Then repeating the argument for $k = 2t + 2$, we obtain~\eqref{tr4} and
finally~\eqref{tr5}. So by Theorem~\ref{t24} for $n=s=2t$ and $H=G_m[V(C)]$, one of three cases below holds:
\medskip

{\bf Case 1:}  $  \overline{K_{t}}+\overline{K_{t}}\subseteq  G_m[V(C)]\subseteq  K_{t}+\overline{K_{t}}$.
As in the proof for $k = 2t + 2$, we derive $G_m=K_m-E(K_{m-t})= H_{m,k,t}$.
\medskip

{\bf Case 2:}  $ \overline{K_{3}}+\ell {K_{t-2}}\subseteq  G_m[V(C)] \subseteq  K_{3}+\ell {K_{t-2}},
$
where $\ell=    
    2(2t-3)/(2t-4)$.
Again, since $G_m[V(C)]$ has hamiltonian cycle $C$ and a separating set of size three
(call this set $A$), $\ell\leq 3$. Since $t \geq 4$, $\ell \neq 3$. If $\ell\leq 2$, then $q\leq 3+2(t-2)<2t$, a contradiction.
\medskip

{\bf Case 3:}  For every two distinct $x,y\in V(C)$, graph  $G_m[V(C)]$ contains a hamiltonian $x,y$-path.
Then for any component $H$ of $G_m-V(C)$, let $x$ and $y$ be neighbors of $H$ in $V(C)$. By the case,
$G_m[V(C)]$ contains a $2t$-vertex path, say $P$. Then $P$ together with an $x,y$-path through $H$ forms
a cycle with at least $k$ vertices, a contradiction. But since $m > k$, such a component $H$ does exist. \qed

\subsection{Subgraphs of $G_m$}

In this section, we define classes of graphs which we shall show are subgraphs of $G_m$, and these subgraphs will have the important
property that they have many long paths and are preserved by the reverse of the contraction process in the Basic Procedure.

\medskip

For a graph $F$ and a nonnegative integer $s$, we denote by $\mathcal{K}^{-s}(F)$ the family of graphs obtained from $F$ by deleting at most $s$ edges.

Let  $F_0=F_0(t)$ denote the complete bipartite graph $K_{t,t+1}$ with partite sets $A$ and $B$ where $|A|=t$ and $|B|=t+1$.
Let $\mathcal{F}_0:=\mathcal{K}^{-t+3}(F_0)$, i.e., the family of subgraphs of $K_{t,t+1}$ with at least $t(t+1)-t+3$ edges.

Let $F_1=F_1(t)$ denote the complete bipartite graph $K_{t,t+2}$ with partite sets $A$ and $B$ where $|A|=t$ and $|B|=t+2$.
Let $\mathcal{F}_1:=\mathcal{K}^{-t+4}(F_1)$, i.e., the family of subgraphs of $K_{t,t+2}$ with at least $t(t+2)-t+4$ edges.

Let $\mathcal{F}_2$ denote the family of graphs obtained from a graph in  $\mathcal{K}^{-t+4}(F_1)$
by subdividing an edge $a_1b_1$ with a new vertex $c_1$, where $a_1\in A$ and $b_1\in B$.
Note that any member $H\in \mathcal{F}_2$ has at least $|A||B|-(t-3)$ edges between $A$ and $B$ and the
 pair $a_1b_1$ is not an edge.

Let $F_3=F_3(t,t')$ denote the complete bipartite graph $K_{t,t'}$ with partite sets $A$ and $B$ where $|A|=t$ and $|B|=t'$.
Take a graph from  $\mathcal{K}^{-t+4}(F_3)$,  select two non-empty subsets $A_1$, $A_2\subseteq  A$ with $|A_1\cup A_2|\geq 3$
 such that $A_1\cap A_2=\emptyset$ if $\min\{ |A_1|, |A_2|\}=1$,
 add two vertices $c_1$ and $c_2$, join them to each other and add the edges from $c_i$ to the elements of $A_i$, ($i=1,2$).
The class of obtained graphs is denoted by ${\mathcal F}(A,B,A_1, A_2)$.
The family $\mathcal{F}_3$ consists of these graphs when $|A|=|B|=t$,
 $|A_1|=|A_2|=2$ and $A_1\cap A_2=\emptyset$. In particular, $\mathcal{F}_3(4)$ consists of exactly one graph, call it $F_3(4)$.

Recall that $F_4$ is a $9$-vertex graph with vertex set $A\cup B$, $A=\{a_1,a_2,a_3\}$, $B:=\{ b_1, b_2, \dots , b_6\}$ and edges
 of the complete bipartite graph $K(A,B)$ and three extra edges $b_1b_2$, $b_3b_4$, and $b_5b_6$.
Define $F_4'$ as the (only) member of ${\mathcal F}(A,B,A_1, A_2)$ where $|A|=|B|=t=4$, $A_1=A_2$, and $|A_i|=3$.
Let $\mathcal{F}_4:= \{ F_4, F'_4\}$, which is defined only for $t=4$.

In this subsection we will prove two useful  properties of graphs
 in $\mathcal{F}_0\cup\dots\cup \mathcal{F}_4$:  First we show that $G_m$ contains one of them (Proposition~\ref{co28}) and
 then show that such graphs
 have long paths with given end-vertices (Lemma~\ref{lem111}).

\begin{prop}\label{co28}
Let $k \geq 9$. If $k$ is odd, then $G_m$ contains a member of $\mathcal{F}_0$, and if $k$ is even then $G_m$ contains a member of $\mathcal{F}_1\cup\dots \cup \mathcal{F}_4$.
 \end{prop}

\begin{proof}
By Proposition \ref{gmprop}, $G_m \subseteq  H_{m,k,t}$ or $m > k = 10$ and $F_4 \subseteq  G_m$.
In the latter case, the proof is complete. So assume  $G_m\subseteq  H_{m,k,t}$ and $A,B,C$ are as in the definition of $H_{m,k,t}$.
First suppose $k$ is even and $C=\{c_1,c_2\}$. If $m=k$ then by~\eqref{j13},
 $$e(H_{m,k,t})-e(G_m) \leq h(m,k,t)-h(m,k,t-1)-1=t - 4,$$ i.e.  $G_m\in\mathcal{K}^{-t+4}(H_{m,k,t})$.
Since $F_1(t)\subseteq  H_{m,k,t}$, $G_m$ contains a subgraph in $\mathcal{F}_1$.
If $m > k$ then by (R2) and Lemma~\ref{l3o}, we have $\delta(G_m) \geq t$. So, each $v\in B$ is adjacent to every $u\in A$ and each of $c_1,c_2$ has
at least $t-1$ neighbors in $A$. Since $|B \cup \{c_1\}|\geq m-t-1\geq t+2$, $G_m$ contains a member of $\mathcal{K}^{-1}(F_1(t))$.
Thus $G_m$ contains a member of $\mathcal{F}_1$ unless $t=4$, $m=2t+3$ and $c_1$ has a nonneighbor $x\in A$. But then $c_1c_2\in E(G_m)$,
and so $G_m$ contains either $F_3(4)$ or $F'_4$.

 Similarly, if $k$ is odd and $m=k$, then  by~\eqref{j13}, $G_m\in\mathcal{K}^{-t+3}(H_{m,k,t})$.
 Thus, since $H_{m,k,t}\supseteq F_0(t)$, $G_m$ contains a subgraph in $\mathcal{F}_0$.
 If $k$ is odd and $m > k$ then by (R2) we have $\delta(G_m) \geq t$. So, each $v\in V(G_m)-A$ is adjacent to every $u\in A$.
 Hence $G_m$ contains $K_{t,m-t}$.
\end{proof}

In order to prove Lemma~\ref{lem111}, we will use Corollary~\ref{zw2} and the following implication of it.

\begin{lem}\label{lem111a}
Let $t\geq 4$ and $H\in {\mathcal F}(A,B,A_1, A_2)$ with $|B|\geq t-1$, $|A|=t$.
Let $P$ be a path $a_1c_1c_2a_2$ and $L$ be a subtree of $H$ with $|E(L)|\leq 2$ such that $P\cup L$ form a linear forest.
Then
\begin{equation}\label{eq22}
   \mbox{\em $H$ has a cycle  $C$ of length $2t+1$ containing $P\cup L$}.
\end{equation}
\end{lem}

{\bf Proof.} Choose some  $B'\subseteq  B$ with $|B'|=t-1$ such that $B\cap V(L)\subseteq  B'$.
 Let $Q$ be the bipartite graph whose $t$-element partite sets are $A$ and $B'\cup \{ c\}$
 where $c$ is a new vertex,
and the edge set  consists of $H[A\cup B']$ and all edges joining $c$ to $A$. By the conditions of the lemma,
the set $E':=E(L)\cup \{a_1c,ca_2\}$ forms a linear forest in $Q$.
Since $Q$ misses at most $t-4$ edges connecting $A$ with $B'\cup\{c\}$, by
 Corollary~\ref{zw2} with $s=t$ and $i=2$, $Q$ has
 a hamiltonian cycle $C'$ containing $E'$. Then the $(2t+1)$-cycle $C$ in $H$ obtained from $C'$
 by replacing path $a_1ca_2$ with $P$ satisfies~\eqref{eq22}.
  \qed

\begin{lem}\label{lem111}
Let $H\in \mathcal{F}_0 \cup \mathcal{F}_1\cup\dots\cup \mathcal{F}_4$ and $x,y\in V(H)$.
\begin{center}
\begin{tabular}{lp{5.5in}}
$(a)$ & $H$ contains an $x,y$-path of length at least $2t-2$;\\
$(b)$ &  if  $H$ does not contain an $x,y$-path of length at least $2t-1$, then \\
& $(b0)$ $H \in \mathcal{F}_0$ and $\{x,y\} \subseteq  A$, or \\
& $(b1)$  $H\in \mathcal{F}_1$ and $\{x,y\}\subseteq  A$, or\\
& $(b2)$  $H=F_4\in \mathcal{F}_4$ and $\{x,y\}\subseteq  A$; \\
$(c)$ & if  $H$ does not contain an $x,y$-path of length at least $2t$, then  \\
& $(c0)$ $H \in \mathcal{F}_0$, or \\
& $(c1)$  $H\in \mathcal{F}_1$ and at least one of $x,y$ is in  $A$, or\\
& $(c2)$  $H\in \mathcal{F}_2$ and either $\{x,y\}\subseteq  A$ or $\{x,y\}=\{a_1,b_1\}$, or\\
& $(c3)$  $H\in \mathcal{F}_3$ and  $\{x,y\}\subseteq  A$, or\\
& $(c4)$  $H\in \mathcal{F}_4$ and $\{x,y\}\subseteq  A$.
\end{tabular}
\end{center}
\end{lem}

{\bf Proof.} The statements concerning $H\in \mathcal F_0 \cup \mathcal F_1$ are the easiest.
Using Corollary~\ref{zw2} (or just using induction on $t$) it is easy to prove a bit more.
Suppose that $H\in {\mathcal K}^{-(t-2)}_{t,t+1}(A,B)$, $t\geq 2$.
Then every pair $x,y\in A\cup B$ is joined by a path of maximum possible length.
This means that every pair of vertices $b_1, b_2\in B$ is joined by a path of length $2t$,
 every pair $a\in A$, $b\in B$ is joined by a path of length $2t-1$, and
 every pair $a_1, a_2\in A$ is joined by a path of length $2t-2$. For example, the proof for $H\in \mathcal F_0$, $a\in A$ and $b\in B$
 is as follows. Consider $H'$ obtained from $H$ by  adding edge $ab$ if $ab\notin E(H)$ and deleting any $b'\in B-b$.
 Then by Corollary~\ref{zw2}, $H'$ has a hamiltonian cycle containing $ab$, which yields an $a,b$-path in $H$ of length $2t-1$.

The cycle $(b_1b_2a_1b_3b_4a_2b_5b_6a_3b_1)$ and path $b_1b_2a_1b_3a_2b_4a_3b_5b_6$
in $F_4$ prove (b2) and the part of (c4) related to $F_4$.

Suppose now that  $H\in \mathcal{F}_2\cup \mathcal{F}_3\cup \{F'_4\}$; even in a more general setting
  suppose that $H\in {\mathcal F}(A,B,A_1, A_2)$ with $|B|=|A|=t$, $|A_1\cup A_2|\geq 3$,
  $|A_2|\geq |A_1|\geq 1$ (and in case of $|A_1|=1$ one has $A_1\cap A_2=\emptyset$).
We prove the statements in reverse order, first (c2) and (c3), then (b), finally (a). When we comment below "Case BC" or "Case AA", this  means
that we consider paths from $B$ to $C$ or from $A$ to $A$, respectively.

By Lemma \ref{lem111a}, we already knew that $c_1c_2$ is contained in a cycle of length $2t+1$ so these two vertices are joined by a path of length $2t$ (Case CC).
If $b\in B$, and $a_i\in A_i$, then the almost complete bipartite subgraph $H[A\cup B]$ contains a $b,a_i$-path of length $2t-1$, so $b$ and $c_{3-i}$ is joined in $H$ by a path of length $2t+1$ (Case BC).
Concerning $b_1, b_2\in B$ we can define $H^+$ by adding an extra vertex $a_{t+1}$ to $A$ and joining it to each vertex of $B$.
Applying Lemma~\ref{lem111a} to $H^+$ (with $t+1$ in place of $t$) we get that it has a cycle $C_{2t+3}$
 through $b_1a_{t+1}b_2$.
This cycle gives a $b_1,b_2$-path of length $2t+1$ in $H$ (Case BB).
In case of $x\in A$, $y\in A$ the high edge density of $H$ implies that $x$ and $y$ have a common neighbor
  $b\in B$.
One can find a path $P=a_1c_1c_2a_2$ such that $P$ and $xby$ form a linear forest.
Then  Lemma~\ref{lem111a}  yields a cycle $C_{2t+1}$ through all these edges. Leaving out $b$
 one gets an $x,y$-path of length $2t-1$ in $H$ (Case AA).
In case of $x\in A$, $y\in B$ maybe we have to add the edge $xy$ to obtain a cycle $C_{2t+1}$ through it by
 Lemma~\ref{lem111a} .
This yields an $x,y$-path of length $2t$ (Case AB).
Finally, if $x\in A$, $y=c_i$ one uses a path $c_i,c_{3-i},x'$ and an $x,x'$-path of length $2t-2$ in $A\cup B$
 to get an $x,y$-path of length $2t$, if this can be done.
If such an $x'\neq x$ does not exists, then $x=a_1\in A_1$, $|A_1|=1$, and $y=c_2$.
This is the case described in (c2) (Case AC).
 \qed

\bigskip

\subsection{Reversing contraction}\label{sec_preserv}

The aim of this section is to prove Lemma~\ref{zl} below on preserving certain subgraphs during the reverse of the Basic Procedure.

\begin{lem}[Main lemma on contraction]\label{zl} Let $k \geq 9$ and suppose $F$ and $F'$ are 2-connected graphs such that $F = F'/xy$ and $c(F') < k$.

If $k$ is even and $F$ contains a subgraph $H \in \mathcal{F}_1\cup\dots \cup \mathcal{F}_4$,
then $F'$ has a subgraph $H' \in \mathcal{F}_1\cup\dots\cup \mathcal{F}_4$.

If $k$ is odd and $F$ contains a subgraph $H \in \mathcal{F}_0$, then $F'$ has a subgraph $H' \in \mathcal{F}_0$.
\end{lem}

{\bf Proof for $k$ even.} {\bf Case 1.} $H\in {\mathcal F}_1$.
Let $u=x* y$. If $u\notin V(H)$ then $H\subseteq  F'$ and we are done. In case of $u\in A$ consider the sets $X:=N_{F'}(x)\cap B$ and $Y:=N_{F'}(y)\cap B$.
If $X=X\cup Y$ then $F'$ restricted to $(A\setminus \{u\}) \cup \{ x\} \cup B$ contains a copy of $H$.
If $X=X\cup Y\setminus \{ y'\}$ for $y' \in V(H')$, then $F'$ restricted to $(A\setminus \{
u\})\cup \{ x\} \cup B\cup \{ y\}$ contains a copy of a graph from $\mathcal F_2$ (with $a_1:=x$, $b_1:=y'$, and $c_1:=y$).
We proceed in the same way if $Y=X\cup Y$ or if $|Y|= |X\cup Y|-1$.
In the remaining case $|X\setminus Y|\geq 2$ and $|Y\setminus X|\geq 2$, so
 one can choose five distinct elements $b_0, x_1, x_2, y_1, y_2$ from $B$ such that
$\{ x_1, x_2\}\subseteq  X\setminus Y$ and $\{ y_1,y_2\}\subseteq  Y\setminus X$.
Then the bipartite subgraph $ Q_0$ of $F'$ generated by the sets $A\setminus \{ u\} \cup \{ x,y\}$ and $B\setminus \{ b_0\}$ contains the
linear forest $L$ consisting of the paths $x_1xx_2$ and $y_1yy_2$.
If we define the graph $ Q$ by adding to $ Q_0$   all edges joining $x$ and $y$ to $B\setminus \{ b_0\}$, then $ Q$ has at least $(t+1)^2-(t-4)$ edges.
So by Corollary~\ref{zw2} for $s=t+1$ and $i=2$, $ Q$ has a hamiltonian cycle $C_{2t+2}$ containing all edges of $L$, and this cycle also appears in $F'$,  contradicting
$c(F')<k$.

In case of $u\in B$ consider the sets $X:=N_{F'}(x)\cap A$ and $Y:=N_{F'}(y)\cap A$.
If $|X\setminus Y|\leq 1$ or $|Y\setminus X|\leq 1$, then we proceed as above and find a subgraph $H'$ of $F$
 either isomorphic to $H$ or belonging to $\mathcal F_2$.
If $|X\setminus Y|\geq 2$ and $|Y\setminus X|\geq 2$, then we have
  four distinct elements $x_1, x_2, y_1, y_2$ in $A$
  such that
$\{ x_1, x_2\}\subseteq  X\setminus Y$ and $\{ y_1,y_2\}\subseteq  Y\setminus X$.
Then $F'$  contains a member of $\mathcal F_3$
 with $(c_1,c_2)=(x,y)$, $A_1:=\{ x_1,x_2\}$, and $A_2:=\{ y_1,y_2\}$.

  \bigskip

 {\bf Case 2.} $H \in \mathcal{F}_2 \cup \mathcal{F}_3 \cup \{F_4'\}$. The proof in this case follows from two claims.
We say that the graph $H$ has the Property ($W_{\ell}$) if the following holds.
\begin{center}
\begin{tabular}{lp{4.5in}}
$(W_{\ell})$ & {\em For all $z \in V(H)$ there exists $w\in N(z)$ such that for all $w'\in N(z) \backslash \{w\}$, the graph $H$ has a cycle $C_{\ell}$ containing the path $wzw'$.}
\end{tabular}
\end{center}

{\bf Claim 1.} {\em Suppose that the graph $F$ contains a subgraph $H$ satisfying Property {\rm ($W_{\ell}$)}, and $c(F')\leq \ell$. Then $F'$ has a subgraph  $H'$ isomorphic to $H$.}

\medskip
Let $z=x*y$ and $V=V(F)-z=V(F')-x-y$.
If $V(H)\subseteq V$, then there is nothing to prove.

Suppose that $z\in V(H)\subseteq V(F)$ and define
 $X:=N_{F'}(x)\cap N_H(z)$ and  $Y:=N_{F'}(y)\cap N_H(z)$.
Then $X\cup Y= N_H(z)$. 
Let $w\in N(z)$ be the vertex from the definition of the Property ($W_{\ell}$). Since $ N_H(z)=X\cup Y$,
we may assume by symmetry that  $w\in X$.

We claim that $Y-w=\emptyset$. Indeed, suppose there is $w'\in Y-w$. By Property {\rm ($W_{\ell}$)},
$H$ has a cycle $C_{\ell}$ containing the path $wzw'$. Then the path $C_{\ell} - z$
in $F'$ together with the edges $w'y$, $yx$ and $xw$ forms a cycle of length $\ell+1$, contradicting $c(F')\leq \ell$.

This implies  that $N_{F'}(x)$ contains $N_H(z)$. So $F'$ contains
a copy of $H$  with the vertex set $(V(H)\setminus \{ z\}) \cup \{ x\}$.  \qed

{\bf Claim 2.} {\em If $H\in \mathcal{F}_2\cup \mathcal{F}_3$ or $H=F'_4$, then $H$ satisfies Property {\rm ($W_{2t+1}$)}.}

\medskip
 We prove a bit more:  every $H\in {\mathcal F}(A,B,A_1, A_2)$ with $|B|\geq t-1$, $|A|=t$
 satisfies ($W_{2t+1}$).
Indeed, for $z=c_i$ we can choose a $w:=c_{3-i}$.
For $z\in B$ we can choose a $w\in A$ arbitrarily.
For $z\in A$ we can choose
 $w\in N(z)\subseteq B$ arbitrarily, except
 if $ z\in A_i$ and $|A_i|=1$.
 In this latter case we can use $w:=c_i$.
In each of these cases, given $L:=wzw'$ one can find a path $P:=a_1c_1c_2a_2$ such that $P\cup L$ is a linear forest.
Then Lemma~\ref{lem111a} yields that $H$ has a cycle $C_{2t+1}$ through $wzw'$.

Since each $H\in \mathcal{F}_2\cup \mathcal{F}_3\cup \{ F'_4\}$ belongs to such ${\mathcal F}(A,B,A_1, A_2)$, this
 completes the proof of Claim~2.   \qed

 \medskip

 {\bf Case 3.} $H = F_4$. Let $u=x*y$. By symmetry, we can consider only two cases: $u=a_1$ and $u=b_1$.
 First, suppose $u=a_1$ and $xb_1\in E(F')$. Then since $c(F')\leq 9$, $y$ is not adjacent to any of $b_3,b_4,b_5,b_6$.
 Thus $x$ is adjacent to all of them, and if $yb_2\in E(F')$, then the cycle $(yb_2b_1a_2b_3b_4a_3b_5b_6xy)$ contradicts
 $c(F')\leq 9$. So $xb_2\in E(F')$ and  the subgraph of $F'$ with vertex set $V(H) \setminus\{u\} \cup \{x\}$ contains $F_4$.

 Similarly, suppose $u=b_1$ and $xb_2\in E(F')$. Then to avoid a $10$-cycle in $F'$, $y$ has no neighbors in $A$
 and thus $x$ is adjacent to all of $A$.
 So, again  the subgraph of $F'$ with vertex set $V(H)\setminus\{u\}\cup\{x\}$ contains $F_4$.

\bigskip

{\bf Proof for $k$ odd.}
First we prove the following statement~\eqref{eq222} which is true for every $t\geq 2$.
Let $H\in {\mathcal K}^{-t+2}(K(A,B))$ with $|A|=t$, $|B|=t+1$.
Let $P$ be a path of length two in $H$. Then
\begin{equation}\label{eq222}
   \mbox{\em $H$ has a cycle  $C$ of length $2t$ containing $P$}.
  \end{equation}
If every vertex of $B\setminus P$ is joined to all vertices of $A$, then one can find a $C_{2t}$ through $P$ directly.
Otherwise, there is a vertex $v\in B\setminus P$ of degree at most $t-1$, so $H\setminus \{ v\}$ is a subgraph of $K_{t,t}$ with at least $t^2-t+3$ edges. Then the statement follows from Corollary~\ref{zw2} for $s=t$ and $i=1$.

Now suppose that $H\in {\mathcal F}_0$, $H \subseteq  F$, $F=F'/xy$, and $H,F$, $F'$ satisfy the constraints of Lemma~\ref{zl}.
Then \eqref{eq222} implies that $H$ satisfies property $(W_{2t})$.
Thus by Claim~1, $F'$ has a subgraph $H'$ isomorphic to $H$.
\qed

\subsection{Completing the proof of Theorem \ref{main}}

{\bf Proof for $k$ even.} Proposition~\ref{co28} and Lemma~\ref{zl} imply that there is a subgraph $H$ of $G=G_n$ such that $H\in \mathcal{F}_1\cup \dots \cup\mathcal{F}_4 $.
Let $G'=G-V(H)$ and $S_1,\ldots,S_s$ be the components of $G'$.
Each of $S_i$ has at least two neighbors, say $x_i$ and $y_i$ in $V(H)$. Let $\ell_i$ denote the length of a longest
$x_i,y_i$-path in $G[V(S_i)\cup \{x_i,y_i\}]$. Since $c(G) < k$, by Lemma~\ref{lem111}(a)~and~(b),
\begin{equation}\label{m28}
\mbox{\em for all $i$,}\quad \ell_i\leq 3 \quad\mbox{\em and if $H\in \mathcal{F}_2\cup \mathcal{F}_3\cup \{ F'_4\}$, then}\quad \ell_i\leq 2.
\end{equation}

\medskip
{\bf Case 1:} $H\in \mathcal{F}_3\cup \{ F'_4\}$. By~\eqref{m28}, $\ell_i\leq 2$ for all $i$ and all choices of $x_i$ and $y_i$.
Since $G$ is $2$-connected, this yields that each $S_i$ is a singleton, say $v_i$. Moreover, Lemma~\ref{lem111}(c3) and (c4)
imply $N(v_i)\subseteq  A$ for all $i$. So $G$ is contained  in a graph in ${\mathcal G}_1(n,k)$,
and the only edge outside $A$ is  $c_1c_2$.

{\bf Case 2:} $H\in \mathcal{F}_2$. Again, by~\eqref{m28}, $\ell_i\leq 2$ for all $i$ and all choices of $x_i$ and $y_i$.
So again this yields that each $S_i$ is a singleton, say $v_i$. But now Lemma~\ref{lem111}(c2)
implies that  for all $i$,  either $N(v_i)\subseteq  A$ or $N(v_i)=\{a_1,b_1\}$.
Thus $G$ is contained  in a graph in ${\mathcal G}_2(n,k)$, where the
only possible star component of $G - A$ with at least three vertices is a star with center $b_1$ and $c_1$ a leaf.

{\bf Case 3:} $H\in \mathcal{F}_1$. Suppose first that some $x_i$ is in $B$. Then by Lemma~\ref{lem111}(c3), $y_i\in A$ and
by Lemma~\ref{lem111}(b), $\ell_i=2$. So, denoting the common neighbor of $x_i$ and $y_i$ in $S_i$ by $c_1$, we get Case~2.
Thus it is enough to consider below only the situation when
\begin{equation}\label{m281}
\mbox{\em
$N(S_i)\cap V(H)\subseteq  A$ for every $i$.}
\end{equation}
We consider three cases.

\medskip

{\em Case 3.1:} For some $i\neq j$, $\ell_i\geq 3$ and $\ell_j\geq 3$, say $\ell_1\geq 3$ and
$\ell_2\geq 3$. Then by~\eqref{m28}, $\ell_1=\ell_2=3$. For $i=1,2$, let  $(x_i,v_i,v'_i,y_i)$ denote an $x_i,y_i$-path of length three in $G[V(S_i)\cup\{x_i,y_i\}]$.
Also, by~\eqref{m281}, $x_1,y_1,x_2,y_2\in A$.
Suppose first that $\{ x_1,y_1\}\neq \{ x_2,y_2\}$.
We proceed as in the beginning of the proof of Lemma~\ref{zl}.
Choose a $(t-2)$-element subset $B'\subseteq B$ and add two new vertices $b_1'$ and $b_2'$ and join them to all vertices of $A$.
Then the obtained bipartite graph $H'$ has at least $t^2-t+4$ edges so there is a hamiltonian cycle $C'$ containing the linear forest
 $x_1b_1'y_1 \cup x_2b_2'y_2$ by Corollary~\ref{zw2}. This $C'$ corresponds to a cycle of length $k$ in $G$, a contradiction.

 It follows that every component $S_i$ with $\ell_i\geq 3$
has exactly two neighbors in $V(H)$ and these two neighbors, say $x_1,y_1$, are the same for all such components; furthermore  $x_1,y_1\in A$.
Furthermore, in order to have $\ell_i \leq 3$, all leaves of $S_i$ have the same neighbor in $A$.
Thus $G$ is contained in a graph in ${\mathcal G}_3(n,k)$.

{\em Case 3.2:} There exists exactly one  $i$ with $\ell_i\geq 3$,  say $\ell_1\geq 3$.
 Then by~\eqref{m28}, $\ell_1=3$. Let $(x_1,v_1,v'_1,y_1)$ be an $x_1,y_1$-path of length $3$ in $G[V(S_i)\cup\{x_1,y_1\}]$.
 By~\eqref{m281}, every other component $S_i$ is a singleton, say $v_i$ with $N(v_i)\subseteq  A$.
 As in Case 3.2, in order to have $\ell_1\leq 3$, $S_1$ should be a star, and if $S_1\neq K_2,K_1$, then all leaves of $S_1$ are adjacent to
the same vertex in $A$. Thus $G$ is contained in a graph in ${\mathcal G}_1(n,k)\cup {\mathcal G}_2(n,k)$.

{\em Case 3.3:} $\ell_i\leq 2$ for all $i$. Here $G$ is contained in a graph in $\mathcal{G}_1(n,k)$. Then each $S_i$ is a singleton with all neighbors in $A$. It follows that $G-A$ is an
independent set.

{\bf Case 4:} $H=F_4$. By Lemma~\ref{lem111}(c4),~\eqref{m281} holds. Together with~\eqref{m28}, this yields that
every component $S$ of $G-A$ is a star and if $|S|\geq 3$, then all leaves of $S$ have the same neighbor in $A$.
It follows that $G\in \mathcal{G}_4(n,k)$.

\bigskip

{\bf Proof for $k$ odd.}
By Proposition~\ref{co28} and Lemma~\ref{zl}, $G_n$ contains some $H\in \mathcal{F}_0$.
 Let $G'=G_n-H$ and $S_1,\ldots,S_s$ be the components of $G'$.
Each of $S_i$ has at least two neighbors, say $x_i$ and $y_i$ in $V(H)$. Let $\ell_i$ denote the length of a longest
$x_i,y_i$-path in $G_n[V(S_i)\cup\{x_i,y_i\}]$. Since $c(G_n)\leq 2t$, by Lemma~\ref{lem111},
\begin{equation}\label{m28'}
\mbox{\em for all $i$,}\quad \ell_i\leq 2\quad\mbox{and $\{x_i,y_i\}\subseteq A$.}
\end{equation}

Then each $S_i$ is a singleton with all neighbors in $A$. It follows that $G-A$ is an
independent set. This completes the proof of Theorem~\ref{main} for $k$ odd. \qed

\section{Proof of Theorem~\ref{t3} for $k \leq 8$}\label{secshort}

Recall that Theorem~\ref{main} describes for $k\geq 9$ and $n\geq 3k/2$ the  $n$-vertex 2-connected graphs  with no cycle of length at least $k$
and more than $h(n,k,t-1)$ edges. In this section, we will do the same for $4\leq k\leq 8$ and $n\geq k$.
We will use for this the classes ${\mathcal G}_i(n,k')$ defined in  Section~\ref{maintheorem} and  the notion of a $J_3$-{\it bridge}.
For $A\subseteq V(G)$ and $S\subseteq V(G)\setminus A$,
 $S$ forms a {\em $J_3$-bridge of $A$ with endpoints $a_1,a_2$} if $a_1, a_2\in A$,
 $A':=\{ a_1, a_2\}$ is a cutset of $G$, $G[S\cup A']\cup \{ a_1a_2\}$ is a $2$-connected graph,  $G[S]$ is connected, and
 the length of the longest $a_1,a_2$-path in $G[S\cup A']$ is three.

Furthermore, since the description (but not the proof) for $k=8$ is more sophisticated,
we will  need four more special graph classes for $k=8$:
Each of the graph classes ${\mathcal G}_i(n,8)$ ($5\leq i\leq 8$) contains $2$-connected $n$-vertex graphs $G$ with $c(G)<8$
and having a special vertex set $A=\{ a_1,a_2, \dots, a_s\}$ with $G[A]$ being a complete graph and such that
 $G\setminus A$ consists of $J_3$-bridges and isolated vertices having exactly two neighbors in $A$.

If $G\in {\mathcal G}_5(n,8)$, then $s=3$ and  $a_1$  is  adjacent to each component in $G\setminus A$.
So the edge $a_2a_3$ is contained in a unique triangle, namely $a_1a_2a_3$.

If $G\in{\mathcal G}_6(n,8)\cup {\mathcal G}_7(n,8)$,  then $s=4$ and the endpoints of all $J_3$-bridges are $\{ a_1, a_2\}$
 while one of the neighbors of some isolated vertex $c$ of $G\setminus A$ is $a_1$ in case of ${\mathcal G}_6(n,8)$ and $N(c)=\{ a_3,a_4\}$ for
 all $c$
  in case of ${\mathcal G}_7(n,8)$.

If $G\in {\mathcal G}_8(n,8)$, then $s=5$ and $N(S)=\{a_1,a_2\}$ for each component $S$ of $G-A$.

\begin{thm}\label{t3small}
Let  $4\leq k \leq 8$ and $n\geq k$. Let $G$ be an  $n$-vertex 2-connected graph  with no cycle of length at least $k$. Then either $7\leq k\leq 8$
and  $e(G)\leq h(n,k,t-1)$ edges or  $G$ is a subgraph of a graph in $\mathcal{G}(n,k)$, where
\begin{center}
\begin{tabular}{lp{5.6in}}
$(1)$ & ${\mathcal G}(n,4)=\emptyset$,\\
$(2)$ & ${\mathcal G}(n,5):={\mathcal G}_1(n,5)$,\\
$(3)$ & ${\mathcal G}(n,6):={\mathcal G}_1(n,6) \cup {\mathcal G}_2(n,6)$,\\
$(4)$ & ${\mathcal G}(n,7):=\{ H_{n,7,3} \}\cup {\mathcal G}_1(n,6)\cup {\mathcal G}_2(n,6)\cup {\mathcal G}_3(n,6)$, \\
$(5)$ & ${\mathcal G}(n,8):= \bigcup_{1\leq i\leq 8, i\neq 4}{\mathcal G}_i(n,8)$.
\end{tabular}
\end{center}
\end{thm}

The proof scheme is that we consider a graph $G$ satisfying the conditions of the theorem and
 take a longest cycle $C$
 with  vertex set, say $X:= \{ x_0, x_1, x_2, \dots,x_r\}$. Moreover, we will assume that
 $C$ has the maximum sum of the degrees of its vertices among the longest cycles in $G$.
 Analyzing possibilities, we will derive that  $G\in\mathcal{G}(n,k)$.

 A  {\em bridge} of $C$ is the vertex set of a component of $G-X$.

We start from a sequence of simple claims on the structure of bridges and the edges between $X$ and the bridges.
For brevity we denote by $d_C(i,j)$ the distance on $C$ between $x_j$ and $x_i$, i.e. $\min\{|j-i|,r+1-|j-i|\}$.
For a bridge $S$ and neighbors $x,x'$ of $S$ on $C$, an $(x,x',S)$-{\em path} is an $x,x'$-path whose all internal vertices
are in $S$.

The maximality of $|C|$ implies our first claim:

\begin{claim}\label{clai1} For every bridge $S$ and any $x_i,x_j\in N(S)\cap X$, the length of any $(x_i,x_j,S)$-path is
at most $d_C(i,j)$. In particular, if $S$ contains distinct $c_1,c_2$ such that $x_ic_1,x_jc_2\in E(G)$, then $d_C(i,j)\geq 3$.
\end{claim}

If $|S|\geq 2$,  then by the $2$-connectedness of $G$, there are two vertex-disjoint $S,X$-paths. Thus if $G[S]$ contains
a cycle, then for some $x_i,x_j\in N(S)\cap X$, the length of the longest $(x_i,x_j,S)$-path is at least $4$. Hence, since
$|C|\leq k-1\leq 7$, by Claim~\ref{clai1}, we get the next claim:

\begin{claim}\label{clai2} For every bridge $S$ of $X$ and any distinct $x_i,x_j\in N(S)\cap X$,
the length of any $(x_i,x_j,S)$-path is at most $3$. In particular,
$G[S]$ is acyclic (a tree).
\end{claim}

Suppose that for some bridge $S$, and two leaves $c_1,c_2$ of the tree $G[S]$, there is a $c_1,c_2$-path $P$ in $G[S]$
of length   at least $3$. Then by Claim~\ref{clai2}, each of $c_1$ and $c_2$ has exactly one neighbor in $X$, and this
is the same vertex, say $x_i$. Again by the $2$-connectedness of $G$, there is $x_j\in X\cap N(S) \setminus\{x_i\}$. Then
there is an $(x_j,x_i,S)$-path of length at least $4$ through either $c_1$ or $c_2$, which contradicts Claim~\ref{clai2}.
Thus we get:

\begin{claim}\label{clai3} For every bridge $S$ of $X$, $G[S]$ is a star. Moreover,
if $|S|\geq 3$, then all leaves of $G[S]$ have degree $2$ in $G$ and the same neighbor, $x(S)$, in $X$.\end{claim}

Suppose $|S|\geq 2$ and $|N(S)\cap X|\geq 3$, say $\{x,x',x''\}\subseteq  N(S)\cap X$. Let $c_1$ be a leaf of $G[S]$.
If $|S|\geq 3$, then by Claim~\ref{clai2} it has a unique neighbor in $X$, say $x$. It follows that there are
an $(x,x',S)$-path and an $(x,x'',S)$-path of length at least $3$. Also there is an $(x',x'',S)$-path of length at least $2$.
Then by Claim~\ref{clai1},  the distance on $C$ from $x$ to $x'$ and to $x''$ is at least $3$ and between $x'$ and $x''$ is
at least $2$. Thus
$|X|\geq 3+3+2=8$, a contradiction. Similarly, if $S=\{c_1,c_2\}$, then by symmetry we may
assume that $x\in N(c_1)\cap X$ and $\{x',x''\}\subseteq  N(c_2)\cap X$. In this case again by Claim~\ref{clai1},
$|X|\geq 3+3+2=8$, a contradiction.
Thus summarizing this with the previous claims, we have proved the following.

\begin{claim}\label{clai4} For every bridge $S$ of $X$ with $|S|\geq 2$, $|N(S)\cap X|=2$. Moreover,
if $|S|\geq 3$, then $G[S]$ is a star and all leaves of $G[S]$ have degree $2$ in $G$ and the same neighbor, $x(S)$, in $X$.
In other words, each bridge $S$ with $|S|\geq 2$ is a $J_3$-bridge of $X$.\end{claim}

From Claims~\ref{clai1} and~\ref{clai4} we deduce:

\begin{claim}\label{clai4a} For every $J_3$-bridge $S$ of $X$ with endpoints $x_i$ and $x_j$,
$d_C(i,j)\geq 3$.\end{claim}

If there are $i_1<i_2<i_3<i_4\leq r$ and bridges $S_1$ and $S_2$ such that $G$ contains an $(x_{i_1},x_{i_3},S_1)$-path $P_1$ and
an $(x_{i_2},x_{i_4},S_2)$-path $P_2$, then we can construct two new cycles $C_1$ and $C_2$ such that each of them contains the
edges of $P_1$ and $P_2$ and each edge of $C$ belongs to exactly one of $C_1$ and $C_2$. Then the total length of $C_1$ and
$C_2$ is at least $|E(C)|+2(|E(P_1)|+|E(P_2)|)\geq (k-1)+8\geq 2k-1$. Thus at least one of them is longer than $C$, a contradiction.
Thus we have:

\begin{claim}\label{clai5} There are no $i_1<i_2<i_3<i_4\leq r$ and bridges $S_1$ and $S_2$ of $X$ such that $G$ contains an $(x_{i_1},x_{i_3},S_1)$-path  and
an $(x_{i_2},x_{i_4},S_2)$-path. In particular, since $k-1\leq 7$, any two $J_3$-bridges share an endpoint.\end{claim}

We  now can prove Theorem~\ref{t3small}. Indeed, by Claim~\ref{clai1}, $|X|\geq 4$. This proves ${\mathcal G}(n,4)=\emptyset$,
i.e., Part 1 of the theorem.

We will consider $3$ cases according to the value of $|X|$. As mentioned above, $|X|\geq 4$.

\bigskip

{\bf Case 1:}  $4\leq |X|\leq 5$. Then by Claims~\ref{clai4} and~\ref{clai4a}, each bridge is a singleton. Furthermore, by Claim~\ref{clai1} each such singleton
has exactly two (necessarily nonconsecutive) neighbors in $X$. If $|X|=4$, Claim~\ref{clai5} yields that this pair of neighbors is the same for all bridges, say
it is $\{x_0,x_2\}$.
Then $G$ is contained in $H_{n,5,2}$ with $A=\{x_0,x_2\}$, as claimed. This proves Part 2.

Let $|X|=5$. If also each bridge has the same pair of neighbors in $X$, say $\{x_0,x_2\}$,  then since  $n\geq |X|+1=6$, $x_1$ is not adjacent to
$\{x_3,x_4\}$ to avoid a $6$-cycle. Thus in this case, $G$ is contained in $H_{n,6,2}$ with $A=\{x_0,x_2\}$, and so $e(G)\leq h(n,6,2)$.
 Otherwise
 by Claim~\ref{clai5}, there are exactly two distinct pairs of neighbors of the bridges, and they share a vertex. Suppose these pairs are
 $\{x_0,x_2\}$ and $\{x_0,x_3\}$ and for $j\in\{2,3\}$, $Y_j$ is the set of vertices adjacent to $x_0$ and $x_j$. Then to avoid a $6$-cycle,
 edges $x_1x_4,x_1x_3$ and $x_2x_4$ are not present in $G$. Then $G\in \mathcal{G}_2(n,6)$ with $A=\{x_0,x_2\}$, $B=Y_2\cup\{x_3\}$ and
 $J=Y_3\cup\{x_4\}$. 
Since $H_{n,6,2}$ contains $H_{n,5,2}$, this together with the previous paragraph proves Part 3 of the theorem.

\bigskip

{\bf Case 2:}  $ |X|=6$.  By Claims~\ref{clai4}--\ref{clai5}, it is enough to consider the following three subcases.

\medskip

{\em Case 2.1:} $X$ has a bridge $S$ with $|N(S)\cap X|\geq 3$.  By Claim~\ref{clai4}, $S$ is a single vertex, say $z$, and by Claim~\ref{clai1},
 $z$ has exactly $3$ (nonconsecutive) neighbors on $C$, say $x_0,x_2$ and $x_4$.
In view of the cycle $x_0zx_2x_3x_4x_5$ and the maximality of the
 degree sum  of $C$, $d(x_1)\geq d(z)\geq 3$. By Claim~\ref{clai5}, $x_1$ has no neighbors outside of $C$. In order to avoid a $7$-cycle in $G$,
 $x_1x_3,x_1x_5\notin E(G)$. So $x_1x_4\in E(G)$.
Similarly, $x_2x_5,x_0x_3\in E(G)$, so $G$ contains $K_{3,4}$ with parts $A=\{x_0,x_2,x_4\}$ and $B=\{x_1,x_3,x_5,z\}$.
Moreover, $B$ is independent. Let $C$ be the vertex set of any component of $G-A-B$. If $C$ has a neighbor in $B$ or is not a singleton,
then $G[A\cup B\cup C]$ has a cycle of length at least $7$. Thus each component of $G-A-B$ is a singleton and has no neighbors in $B$.
This means
 $A $ meets all edges and so $G$ is a subgraph of $H_{n,7,3}$.
\medskip

{\em Case 2.2:}  $X$ has a $J_3$-bridge $S$.  Then by Claim~\ref{clai1} and symmetry, we may assume  $N(S)=\{ x_0,x_3\}$.
 In this case, $G$ has $3$ internally disjoint $x_0,x_3$-paths of length $3$. Thus to have $c(G)\leq 6$,
  $\{ x_0,x_3\}$ separates internal vertices of distinct paths. It
  follows that $G-\{x_0,x_3\}$ is a collection of $J_3$-bridges of  $\{x_0,x_3\}$ and isolated vertices
  each having only $x_0$ and $x_3$ as endpoints.
Thus $G$ is a subgraph of a graph in  ${\mathcal G}_3(n,6)$. 

 \medskip

{\em Case 2.3:} $V\setminus X$ is independent and each $z\in V\setminus X$ has  degree $2$.
 By Theorem~\ref{th:Kopylov2}, for each $z\in V\setminus X$, graph $G[X
\cup \{ z\}]$ has at most $h(7,7,2)=14$ edges, which yields $e(G)\leq 2n=h(n,7,2)$.
This proves Part 4 of Theorem~\ref{t3small}.

\bigskip

{\bf Case 3:}  $ |X|=7$. By Claims~\ref{clai4}--\ref{clai5}, it is enough to consider the following four subcases.

\medskip

{\em Case 3.1:} $X$ has a bridge $S$ with $|N(S)\cap X|\geq 3$.  As in Case 2.1, $S$ is a single vertex, say $z$, and
we may assume $N(S)\cap X=\{x_0,x_2,x_4\}$. Again, similarly to Case 2.1, in view of
 the $7$-cycle $x_0zx_2x_3x_4x_5x_6$, we obtain that $d(x_1)\geq d(z)\geq 3$, and  that (to avoid a long cycle in $G$) the third neighbor of $x_1$  is $x_4$.
Similarly, $x_0x_3\in E(G)$. Thus,  $G$ has a subgraph consisting of   $K_{3,3}$ with parts $A:=\{x_0,x_2,x_4 \}$ and $B:=\{ x_1, x_3, z\}$ and an attached $3$-path $x_4x_5x_6x_0$. Moreover, $d(x_1)=d(x_3)=d(z)=3$ and these are isolated vertices in $G\setminus A$. Let $Y$ be the vertex set of the component of $G-A$ containing
$\{x_5,x_6\}$. If there is another component $Y'$ of $G-A$ with $|Y'|\geq 2$, then to avoid a $\geq 8$-cycle, $G$ must be a subgraph of a graph in  ${\mathcal G}_3(n,8)$.
If all the bridges of $A$ apart from $A$ are singletons, then $G$ is a subgraph of a graph in either ${\mathcal G}_1(n,8)$  (if $|Y|=2$) or ${\mathcal G}_2(n,8)$  (if $|Y|\geq 3$).

\medskip

{\em Case 3.2:} $G$ has $J_3$-bridges $S_1$ and $S_2$ of $X$ with $N(S_1)\neq N(S_2)$. By Claims~\ref{clai5} and~\ref{clai4a},
we may assume $N(S_1)=\{x_0,x_3\}$ and $N(S_2)=\{x_0,x_4\}$. By the $2$-connectivity of $G$, we may assume that there is an
$(x_0,x_3,S_1)$-path
   $x_0y_1y_2x_3$ and an $(x_4,x_0,S_2)$-path $x_4y_5y_6x_0$. Let $A=\{ x_0,x_3,x_4\}$.
Then the edges $y_1y_2$, $y_5y_6$, $x_1x_2$, $x_5x_6$ belong to distinct components of $G\setminus A$.
Thus to avoid long cycles in $G$, no bridge of $A$ is adjacent to both, $x_3$ and $x_4$ and none of the bridges $S$ of $A$
contains an $(x_0,x_3,S)$-path or an $(x_0,x_4,S)$-path of length at least $4$. It follows that $G$ is a subgraph of a graph in
${\mathcal G}_5(n,8)$.

\medskip

{\em Case 3.3:} $G$ has a $J_3$-bridge $S$  of $X$, and every other $J_3$-bridge of $X$ (if exists) has the same neighbors as $S$
in $X$.  We may assume that $N(S)\cap X=\{x_0,x_4\}$ and $G$ contains an $(x_0,x_4,S)$-path
 $x_0y_6y_5x_4$.
Then the edges $y_5y_6$, $x_1x_2$, $x_5x_6$ belong to three distinct components of $G\setminus \{ x_0,x_4\}$.
Let $Y$ be the component  of $G\setminus \{ x_0,x_4\}$ containing $\{ x_1,x_2,x_3\}$.
By the case, all other components are either isolated vertices or $J_3$-bridges of $\{ x_0,x_4\}$.
Also, every vertex $y\in (Y\setminus \{ x_1,x_2,x_3\}) $ has only neighbors in $X$ (i.e., $N(y)\subset \{ x_0, x_1, \dots, x_4\} $).

If $|Y|=3$ we obtain that $G$ is a subgraph of a member of ${\mathcal G}_8(n,8)$ with $A=\{x_0,x_1,x_2,x_3,x_4\}$.
Suppose $|Y|\geq 4$. If there is $y\in Y \setminus\{x_1,x_2\}$ with $N_G(y)=\{x_0,x_3\}$, then to avoid an $8$- or $9$-cycle,
 $x_1x_4\notin E(G)$ and no $y'\in Y\setminus\{x_2,x_3\}$ has $N_G(y')=\{x_1,x_4\}$. So, either $\{x_0,x_3\}$ is a cut
 set in $G$ or $x_2x_4\in E(G)$. In the former case, $G$ is a subgraph of a graph in
${\mathcal G}_5(n,8)$ with $A=\{x_0,x_3,x_4\}$ and $a_1=x_0$. In the latter case, in order to avoid an $(x_0,x_4,Y)$-path of length
$\geq 5$, graph $G[\{x_1,x_2,x_3,x_4,y\}]$ has only the $5$ edges we already know and no vertex $y'\in Y-X-y$ has
$N(y')\subseteq \{x_1,x_2,x_3,x_4,y\}$. This means $G$ is a subgraph of a graph in
${\mathcal G}_6(n,8)$ with $A=\{x_0,x_4,x_2,x_3\}$, where $a_1=x_0$ and $a_2=x_4$. The case of
$y\in Y\setminus\{x_1,x_2\}$ with $N_G(y)=\{x_1,x_4\}$ is symmetrical. If there is $y\in Y\setminus\{x_1\}$ with $N(y)=\{x_0,x_2\}$, then
in order to avoid an $(x_0,x_4,Y)$-path of length
$\geq 5$, $x_1x_3\notin E(G)$ and every $y'\in Y-X$ is adjacent to $x_2$. This means $G$ is a subgraph of a graph in
${\mathcal G}_2(n,8)\cup {\mathcal G}_3(n,8)$ with $A=\{x_2,x_4,x_0\}$. The last possibility is that $N(y)=\{x_1,x_3\}$
for every $y\in Y-X$. Since $|Y|\geq 4$, this yields $x_2x_0,x_2x_4\notin E(G)$. Thus
 $G$ is a subgraph of a member of ${\mathcal G}_7(n,8)$
 with $\{a_1,a_2\}:=\{ x_0,x_4\}$ and $\{a_3,a_4\}:=\{x_1,x_3\}$.

\medskip

{\em Case 3.4:} $G\setminus X$ consists of isolated vertices only, each having degree 2 in $G$.
By Theorem~\ref{th:Kopylov2}, for each $z\in V\setminus X$, graph $G[X
\cup \{ z\}]$ has at most $h(8,8,2)=19$ edges, which yields $e(G)\leq 2n+3=h(n,8,2)$.
\qed

\bigskip

Theorem~\ref{t3small} yields the following analog of Theorem~\ref{main}(1) for a smaller range of $e(G)$.

\medskip

\begin{cor}\label{cor_C7}
Suppose that $G$ is a $2$-connected, $n$-vertex graph with $c(G)<7$,  $n\geq 8$.
If $e(G)\geq \lfloor (5n-6)/2\rfloor$ then $G$ is a subgraph of $H_{n,7,3}$, and this bound is  best possible.
 \qed
 \end{cor}

 \bigskip

\section{ Concluding remarks}
It could be that for $k\geq 11$, Theorem~\ref{t3} holds already for $n\geq 5k/4$. Note that by Theorem~\ref{th:Kopylov2}, it does not hold
for $n<5k/4$. It may also be possible, albeit complicated, to describe the structure of 2-connected $n$-vertex graphs with no cycles of length at least $k = 2t + 1$ and at least $h(n,k,t - 2)$ edges. We leave these as avenues for further research.

\paragraph{Acknowledgment.}
We thank both referees and R. Luo for very helpful comments.


\begin{thebibliography}{99}



\bibitem{BonChv}
J. A. Bondy and V. Chv\'atal,
A method in graph theory.
Discrete Math. {\bf 15} (1976), 111--135.

\bibitem{BonJac}
J. A. Bondy and B. Jackson,
Long paths between specified vertices of a block, Ann. Discrete Math. {\bf 27} (1985), 195--200.


\bibitem{Ch}
V. Chv\'atal,
On Hamilton's ideals.
J. Combinatorial Theory Ser. B {\bf 12} (1972), 163--168.

\bibitem{Dir} G. A. Dirac,
Some theorems on abstract graphs,
Proc. London Math. Soc. (3) {\bf 2}, (1952). 69--81.

\bibitem{En}
H. Enomoto,
Long paths and large cycles in finite graphs,
{ J. Graph Theory} {\bf 8} (1984), 287--301.

\bibitem{Erd62}
P. Erd\H{o}s,
Remarks on a paper of P\'osa,
Magyar Tud. Akad. Mat. Kutat\'o Int. K\"ozl. {\bf 7} (1962), 227--229.


\bibitem{ErdGal59}
P. Erd\H{o}s and T. Gallai,
On maximal paths and circuits of graphs,
Acta Math. Acad. Sci. Hungar. {\bf 10} (1959), 337--356.


\bibitem{FaudScheB}
R. J. Faudree and R. H. Schelp, Ramsey type results,
Infinite and Finite Sets, {Colloq. Math. J. Bolyai} {\bf 10}, (ed. A. Hajnal et al.), North-Holland,
Amsterdam, 1975, pp. 657--665.

\bibitem{FaudSche75}
R. J. Faudree and R. H. Schelp,
Path Ramsey numbers in multicolorings,
J. Combin. Theory Ser. B {\bf 19} (1975), 150--160.


\bibitem{FS224}
Z. F\"uredi and M. Simonovits,
  The history of degenerate (bipartite) extremal graph problems,
Bolyai Math. Studies {\bf 25} pp. 169--264,
Erd\H{o}s Centennial (L. Lov\'asz, I. Ruzsa, and V. T. S\'os, Eds.) Springer, 2013.
Also see: {\tt arXiv:1306.5167}.



\bibitem{Kopy}
G. N. Kopylov,
Maximal paths and cycles in a graph,
Dokl. Akad. Nauk SSSR {\bf 234} (1977),  19--21.
(English translation: Soviet Math. Dokl. {\bf 18} (1977), no. 3, 593--596.)


\bibitem{Lewin}
M. Lewin,
On maximal circuits in directed graphs.
J. Combinatorial Theory Ser. B {\bf 18} (1975), 175--179.


\bibitem{Lovasz} L. Lov\'asz, Combinatorial Problems and Exercises, $2^{nd}$ Ed., North-Holland, Amsterdam, 1993.

\bibitem{Posa62}
L. P\' osa,
A theorem concerning Hamilton lines,
Magyar Tud. Akad. Mat. Kutat\'o Int. K\"ozl. {\bf 7} (1962), 225--226.


\bibitem{Po} %
L. P\' osa,
On the circuits of finite graphs,
Magyar. Tud. Akad. Mat. Kutat\' o Int. K\H ozl. {\bf 8} (1963/1964), 355--361.

\bibitem{Woodall}
D. R. Woodall,
Maximal circuits of graphs I,
Acta Math. Acad. Sci. Hungar. {\bf 28} (1976), 77--80.

\bibitem{ZW}
R. Zamani and D. B. West,
Spanning cycles through specified edges in bipartite graphs, { J. Graph Theory} {\bf 71} (2012),  1--17.



\end{thebibliography}
\end{document}